\documentclass[reqno]{amsart}
\usepackage{amssymb}
\usepackage{amscd,enumerate,amsfonts,calc,amsmath,verbatim, euscript}
\usepackage{enumerate}
\usepackage{enumitem}
\usepackage{tikz}
\usepackage{xcolor}
\usepackage[hidelinks]{hyperref}
\hypersetup{colorlinks,breaklinks,
             urlcolor=blue,
             linkcolor=blue}

\newcommand{\A}{\mathbb{A}}
\newcommand{\PP}{\mathbb{P}}

\newcommand{\N}{\mathbb{N}}
\renewcommand{\k}{\Bbbk}
\newcommand{\Fq}{\mathbb{F}_q}

\newcommand{\F}{\mathbb{F}}
\newcommand{\RM}{\mathrm{RM}}
\newcommand{\PRM}{\mathrm{PRM}}
\newcommand{\Supp}{\mathrm{Supp}}
\newcommand{\wt}{\mathrm{wt}}
\newcommand{\rk}{\mathrm{rk}}

\usepackage{booktabs,diagbox}
\usepackage{xcolor}

\newtheorem{theorem}{{\bf Theorem}}
\newtheorem{lemma}[theorem]{{\bf Lemma}}

\newtheorem{remark}[theorem]{{\bf Remark}}
\newtheorem{proposition}[theorem]{{\bf Proposition}}
\newtheorem{corollary}[theorem]{{\bf Corollary}}
\newtheorem{definition}[theorem]{{\bf Definition}}
\newtheorem{ex}[theorem]{{\bf Example}}
\newtheorem{example}[theorem]{{\bf Example}}

\makeatletter
\newcommand{\addresseshere}{%
  \enddoc@text\let\enddoc@text\relax
}
\makeatother

\begin{document}
\title[Betti Numbers and Higher Weight Spectra of $\RM_q(2,2)$]{Betti Numbers and Higher weight spectra of Reed-Muller codes $\RM_q(2,2)$}
\author[S. R. Ghorpade]{Sudhir R. Ghorpade}
	\address{Department of Mathematics, 
		Indian Institute of Technology Bombay,\newline \indent
		Powai, Mumbai 400076, India}
\email{\href{mailto:srg@math.iitb.ac.in}{srg@math.iitb.ac.in}}
\thanks{Sudhir Ghorpade is partially supported by  the grant DST/INT/RUS/RSF/P-41/2021 from the Department of Science \& Technology, Govt. of India, and the Indo-Norwegian Cooperation Programme 2024 (INCP-2) from the University Grants Commission, Govt. of India.}

\author[T. Johnsen]{Trygve Johnsen}
\address{Department of Mathematics and Statistics, UiT: The Arctic University of Norway, \newline \indent
N-9037 Troms\o, Norway}
\email{\href{mailto:Trygve.Johnsen@uit.no}{Trygve.Johnsen@uit.no}}
\thanks{Trygve Johnsen was partially supported by the project Pure Mathematics in Norway, funded by Bergen Research Foundation and Troms\o\  Research Foundation and the UiT-project Mascot. He is also partially supported by the Indo-Norwegian Cooperation Programme 2024 (INCP-2) from the Norwegian Directorate for Higher Education and Skills.} 
    
\author[R. Ludhani]{Rati Ludhani}
\address{Inria Saclay Centre, 1 Rue Honore d'Estienne d'Orves, 91120 Palaiseau, France}
\email{\href{mailto:ratiludhani@gmail.com}{ratiludhani@gmail.com}}
\thanks{Rati Ludhani was supported by the Prime Minister's Research Fellowship PMRF-192002-256, and is currently supported by the European Union's Horizon 2020 research and innovation program under the Marie Sk{\l}odowska-Curie grant agreement n$^\circ$101034255. }
\author[R. Pratihar]{Rakhi Pratihar} 
\address{Institut de Recherche Math\'ematique de Rennes (IRMAR), UMR CNRS 6625, Universit\'e de Rennes - Campus Beaulieu, 
35700 Rennes, France}
\email{\href{mailto:pratihar.rakhi@gmail.com}{pratihar.rakhi@gmail.com}}
\thanks{Rakhi Pratihar was partially funded by French Agence Nationale de la Recherche in the context of Plan France 2030 with reference ANR-22-PETQ-0008. 
}


\begin{abstract}
We determine all the Betti numbers of the $q$-ary second order Reed-Muller codes of length $q^2$, and also of the elongations of matroids associated to these codes. We then use it  to determine the higher weight spectra of these codes. As a  special case, we recover some results of Kaplan and Matei about counting certain curves over finite fields with prescribed rational intersection points.  In geometric terms, our results relate to the affine Veronesean by which we mean the image of the affine plane $\A^2$ under the quadratic Veronese embedding of $\PP^2$ in $\PP^5$. Indeed, finding the higher weight spectra,  of the Reed-Muller code considered here,  corresponds to determining the number of $\Fq$-rational points on all possible sections of this affine Veronesean by linear subvarieties of $\PP^5$. 
\end{abstract}


\maketitle

%
%
\section{Introduction}
Let $n$ be a positive integer and let $C$ be a linear error-correcting code of length $n$, i.e., a subspace of the $n$-dimensional vector space $\Fq^n$ over the finite field $\Fq$ with $q$ elements. 
 The study of invariants of $C$ is important  since a fundamental problem in coding theory is to distinguish two inequivalent codes. 
Betti numbers of  $C$ are a fine set of invariants of $C$, which were introduced a little over a decade ago in \cite{JV} and have been of some current interest (see, e.g, \cite{GS, GR, JPR}). These are defined by considering the vector matroid $M_C$ corresponding to a  parity check matrix of $C$ and then an $\N$-graded minimal free resolution of the Stanley-Reisner ring of the simplicial complex formed by independent sets of $M_C$. (See Section~\ref{matroids} for details.) Determining all the Betti numbers is usually a difficult problem and a solution is known only for a few classes of linear codes. But this can be interesting and useful. To begin with, it was shown in \cite{JV} that the Betti numbers of $C$ readily yield the generalized Hamming weights of $C$. Furthermore, by combining the results in \cite{JRV} and \cite{K1}
 (or \cite{JP}), we see that the Betti numbers of $M_C$ and of its elongations determine  the higher weight spectra of $C$.  It may be noted, however, that 
 two linear codes can have the same higher weight spectra, but different Betti numbers. (See Example \ref{bettiremark}, and also Remark \ref{newphi}.) 
 Thus, the Betti numbers of a code are finer set of invariants of the code (and also of the associated matroid) than the classical invariants such as  weight distribution, generalized Hamming weights, and higher weight spectra. In other words,  the Betti numbers bear more structural information about a linear code. It is conceivable that this can be exploited further for cryptographic applications; see, e.g., the 
 recent work of Randriambololona \cite{R}.

The notion of \emph{higher weight spectra}, also known as \emph{support weight distribution}, is more classical and it can be traced 
back at least 
to the work of 
Helleseth,  Kl{\o}ve and Mykkeltveit \cite{HKM} in 1977 (see also \cite{K1, K2}).  For $C$ as above, finding the higher weight spectra of $C$ amounts to determining for each $w = 0, 1, \dots, n$ and $r=0,1, \dots, \dim C$,  the number $A_w^{(r)}(C)$ of $r$-dimensional subcodes of $C$ of support weight $w$. The \emph{spectrum}, or the \emph{weight distribution}, of $C$ is a special case and corresponds essentially to $r=1$. The determination of the \emph{generalized Hamming weights}, also known as the \emph{higher weights}, $d_r(C)$ of $C$ is  a special case as well, since $d_r(C)$ is the least nonzero value of $w$ for which $A_w^{(r)}(C)\ne 0$. The problems of determining the spectrum or the generalized Hamming weights have been considered in the literature for many classes of linear codes with varying degrees of success. However, the complete determination of higher weight spectra appears to be known in far fewer cases. 

The class of codes we consider here is that of $q$-ary (or generalized) \emph{Reed-Muller codes} $\RM_q(d,m)$ of order $d$ and length $q^m$. It is among the most widely studied classes of linear codes and it has been of significant theoretical and practical interest. These codes were introduced, in the binary case 
(i.e., when $q=2$), by 
D. E. Muller and I. S. Reed in 1954.
The study of $q$-ary Reed-Muller codes was   pioneered in the late 1960's by Kasami-Lin-Peterson \cite{KLP}  and Delsarte-Goethals-MacWilliams \cite{DGM}. In the binary case, the weight distribution of any  second order code $\RM_2(2,m)$ was given in 1970 by Sloane and Berlekamp \cite{SB}. The weight distribution of  $\RM_2(3,m)$ is not known, in general,  and it has been of much current interest. For many partial results in this direction, we refer to the works of 
Carlet and Sol\'e \cite{CS}, Carlet \cite{C}, Markov and Borrisov \cite{MB}, Lou and Wang \cite{LW}, and Rameshwar, Jain and Kashyap \cite{JRK}, as well as references therein. 
In the $q$-ary case, a weight distribution of $\RM_q(2,m)$ was proposed by McEliece \cite{Mc} in 1969, but it appears that there were some mistakes in the computations and a precise version was given in 2019 by Li \cite{Li}. 
The generalized Hamming weights of any $\RM_q(d,m)$ were determined 
by Heijnen and Pellikaan \cite{HP} in 1998 (see also the related works of Beelen and Datta \cite{BD} and Beelen \cite{B}).  
For the general problem  of the determination of higher weight spectra of $\RM_q(d,m)$, the best known result thus far appears to be that of Jurrius \cite{J} where the case of first order, i.e., $d=1$, has been resolved. 
For a related class of linear codes, known as \emph{Veronese codes} or \emph{projective Reed-Muller codes} $\PRM_q(d,m)$, higher weight spectra has been determined in the following cases: $\PRM_q(2,2)$ and $\PRM_2(2,3)$ by Johnsen and Verdure \cite{JV20, JV21}, and 
$\PRM_3(2,3)$ by Kaipa and Pradhan \cite{KP}. 

Our main results in this paper are the complete determination of all the Betti numbers and  the higher weight spectra of the $q$-ary second order Reed-Muller code $C_q:= \RM_q(2,2)$ of length $q^2$. 
The proof is motivated by the methods used in \cite{JV20} and the strategy is explained in Section~\ref{sec3}. In particular, we not only determine the Betti numbers of $C_q$, but also those of all the elongations of the matroid  associated to ${C_q}$. It may be noted that \cite{JV20} uses a classification of conics in the projective plane over $\Fq$. Such a classification is known in the literature on finite geometry (see, e.g., \cite{Hi}).  We attempt a similar approach, which requires a classification of conics in the affine plane over $\Fq$. But this doesn't appear to be readily available and seems rather involved. For instance, while conics in $\PP^2(\Fq)$ can be one of 4 classes, the conics in $\A^2(\Fq)$ get subdivided into 10 classes. A precise result is given in  Proposition~\ref{conicvariations}, and this may be of independent interest. 

Questions about the higher weight spectra, and in particular, about generalized Hamming weights and the weight distribution, have a well-known geometric interpretation using, for instance, the language of projective systems of Tsfasman and Vl\u{a}du\c{t} \cite{TV}. From a geometric viewpoint, $C_q= \RM_q(2,2)$ corresponds to the image of the affine plane $\A^2$ (over $\Fq$) under the quadratic Veronese embedding of $\PP^2$ in $\PP^5$. One may call this image as an \emph{affine Veronesean} and denote by $\mathcal{V}^{\A}_{2,2}$. Thus finding $A_w^{(r)}(C_q)$ corresponds to finding the number of linear subspaces $L_r$ of $\PP^5$ of codimension $r$ for which $\left| \mathcal{V}^{\A}_{2,2} \cap L_r\right| =w$. Alternatively, this corresponds to determining the number of $\Fq$-rational points on intersections of $r$ ``distinct" affine conics. Such questions were investigated by Kaplan and Matei \cite{KM} for intersections of two (or less!) conics, and they can be readily recovered from our determination of higher weight spectra of $C_q$. See Remark~\ref{four} for details. 

For the convenience of the reader, a complete statement of the result concerning the higher weight spectra of the code $C_q= \RM_q(2,2)$ is given in Section~\ref{sec2} below, while the proof is given in Section \ref{sec8} for $q\ge 7$ and in Section~\ref{smallq} for $q\le 5$. The result about the Betti numbers of $C_q$ and of the elongations of the matroid associated to $C_q$ is stated and proved in Sections \ref{minres} and \ref{minreselong} for $q\ge 7$, while these Betti numbers are tabulated in the Appendix for the remaining values of $q$, i.e., for $q=2,3,4, 5$. 

\section{The Code $C_q$ and its Higher Weight Spectra}
\label{sec2}

Throughout this paper, we let $q$ be a prime power and $\Fq$ the finite field with $q$ elements. For any nonnegative integer $n$, we denote by $\PP^n_q$ the $n$-dimensional projective space over $\Fq$. 
Following \cite{JV20}, we let $\nu_q$ be the Veronese  map that maps $\mathbb{P}^2_q$ into $\mathbb{P}^5_q$, 
i.e., $\nu_q$ maps $(x:y:z)$ 
to  $(x^2:xy:xz:y^2:yz:z^2)$, and let $V_q$ be the image of $\nu_q$. Then $V_q$  is the set of $\Fq$-rational points of a non-degenerate smooth surface of degree $4$. The cardinality $|V_q|$ of $V_q$ is  $|\mathbb{P}^2_q|= q^2+q+1$. Fix some order for the points of 
$V_q$, and for each such point, fix a coordinate $6$-tuple that represents it.
For 
the points corresponding to $(x:y:z)$ with $z\ne 0$, we always pick the representative with $z=1.$
Let $G_q$ be the $(6 \times q^2)$ matrix, whose columns are the coordinate $6$-tuples  of the points of $V_q$, taken
in the fixed order,  deleting the $q+1$ columns with $z=0$. The deletion of these columns is the essential difference between the codes studied here and the ones in \cite{JV20}, and, as we will see, it makes the 
problem more complicated to study.
\begin{definition}
{\rm For $q\ge 3$, the $q$-ary \emph{Reed-Muller code} $C_q = \RM_q(2,2)$ is the linear code
with generator matrix $G_q$. For $q=2$, the \emph{Reed-Muller code} $C_2 = \RM_2(2,2)$ is the linear 
code with generator matrix $G_2'$, where $G_2'$ is the $(4\times q^2)$ matrix obtained from $G_2$ by removing 
the third and fifth rows, 
which correspond to $xz$ and $yz$, respectively.}
\end{definition}

\begin{remark}
{\rm The first $2$ in $\RM_q(2,2)$ refers to the maximal degree $2$ of the polynomials in $x,y$ that are evaluated to obtain entries in $G_q$, and the second $2$ refers to the number of variables $x,y$ that appear, which is $2$ after we have set $z=1$ everywhere. It is elementary and well-known that $C_q$ is a 
$[q^2, \, 6,\, q^2-2q]$-code if $q\ge 3$ and $C_2$ is a binary $[4,  4,  1]$-code; in fact, $C_2 = \mathbb{F}_2^4$.
See, for example, \cite{DGM}.
}
\end{remark}


\begin{definition}
{\rm Let $n$ be a positive integer and let $k$ be a nonnegative integer with $k\le n$. For a linear $[n,k]_q$-code $C$  in $\F_q^n$ and for a codeword  ${\bf w}=(w_1,\cdots,w_n) \in C$, let $\Supp({\bf w})$ be the set of indices $i\in \{1,\ldots,n\}$ where $w_i \ne 0.$ For a set $T$ of codewords, let $\Supp(T)$ be the union of the 
$\Supp({{\bf w}})$, for ${\bf w} \in T$. Let $\wt({\bf w})$ be the cardinality of $\Supp({{\bf w}})$, and $\wt(T)$ the cardinality of $\Supp(T).$}
\begin{itemize}
{\rm \item For each $r = 0,1,\ldots,k$, let $d_r(C)$ be the smallest nonnegative integer $w$ such that there exists a $r$-dimensional 
subcode $D$ of $C$ with $\wt(D) = w.$ We call $d_r(C)$ the $r$-th \emph{generalized Hamming weight} or the $r$-th \emph{higher weight} of $C$. Note that $d_0(C)=0$. 
\item For each $r=0,1,\ldots,k$ and for each $w = d_r(C),\ldots,n$, let $A_w^{(r)}(C)$ be the number of  linear subcodes $D$ of dimension $r$ of $C$, with $\wt(D) = w.$
(If $w$ is outside the range $[d_r(C),n]$ we set $A_w^{(r)}(C)=0.)$
For a fixed $r \in \{0, 1, \ldots, k\}$, the ordered multiset $\{A_w^{(r)}(C)\},$ for $w=d_r(C),\ldots,n$, is called the $r$-th \emph{weight spectrum} of $C$. We call the multiset of $r$-th 
weight spectra for $r=0, 1,\dots,k$ as the \emph{higher weight spectra} of $C$. Note that $A_0^{(0)}(C)=1$. 

\item For $0\le r\le k$, we denote by $W_r(X,Y)$ the polynomial 
 $$
 \sum_{w=0}^{n} A^{(r)}_w(C) X^{n-w}Y^w. 
 $$
 This is 
 called the $r$-th \emph{higher weight polynomial} of the code $C$. }
\end{itemize}
\end{definition}

%
%
%
One of our main results is the following complete description of the higher weight spectra for the code $C_q$ for all prime powers $q$. 
\begin{theorem} \label{Fallq}
Consider the Reed-Muller code $C_q = \RM_q(2,2)$. Let $A_w^{(r)} = A_w^{(r)}(C_q)$ for $w\ge 0$ and $0\le r \le \dim C$. Then: 
\begin{itemize}
\item[{$(a)$}] 
If $q\ge 7$, then for the $[q^2, \, 6,\, q^2-2q]$-code $C_q$, we have 
\begin{eqnarray*}
&& A_0^{(0)}=1, \quad A_{q^2-2q}^{(1)}=\frac{q^3-q}{2},\quad A_{q^2-2q+1}^{(1)}=\frac{q^4+q^3}{2},\\
&&\vspace{0.01mm}\\
&& A_{q^2-q-1}^{(1)}=\frac{q^5-2q^4+q^3}{2},\quad A_{q^2-q-1}^{(2)}=q^4-q^2, \\
&&\vspace{0.02mm}\\
&& A_{q^2-q}^{(1)}=q^4+q^2+2q,\quad A_{q^2-q}^{(2)}=2q^3 + 3q^2 + q,\quad A_{q^2-q}^{(3)}=q^2 + q,\\
&&\vspace{0.01mm}\\
&& A_{q^2-q+1}^{(1)}=\frac{q^5-q^3}{2},\\
&&\vspace{0.01mm}\\
&& A_{q^2-4}^{(2)}= \frac{q^8-4q^7+5q^6+q^5-6q^4+3q^3}{24},\\
&&\vspace{0.01mm}\\
&& A_{q^2-3}^{(2)}=\frac{4q^7-9q^6+q^5+9q^4-5q^3}{6},\quad A_{q^2-3}^{(3)}=\frac{q^6-q^5-q^4+q^3}{6}, 
\end{eqnarray*}
\begin{eqnarray*}
&& A_{q^2-2}^{(2)}=\frac{q^8-2q^7+13q^6-9q^5-14q^4+11q^3}{4},\quad A_{q^2-2}^{(3)}= \frac{q^7+q^5-2q^3}{2},\\
&&\vspace{0.01mm}\\
&& A_{q^2-2}^{(4)}= \frac{q^4-q^2}{2},
\\  &&\vspace{0.01mm}\\
&& A_{q^2-1}^{(1)}=\frac{q^4-q^3}{2},\quad A_{q^2-1}^{(2)}=\frac{2q^8+4q^7-5q^6+29q^5+15q^4-27q^3+6q^2}{6},\\
&&\vspace{0.01mm}\\
&& A_{q^2-1}^{(3)}=\frac{2q^8+3q^6+3q^5+5q^4+3q^3}{2},\ A_{q^2-1}^{(4)}=q^6 + q^5 + q^3 + 2q^2,\ A_{q^2-1}^{(5)}=q^2,\\
&&\vspace{0.01mm}\\
&& A_{q^2}^{(1)}=\frac{q^3-q+2}{2},\ A_{q^2}^{(2)}=\frac{9q^8+8q^7+21q^6-19q^5+42q^4+59q^3-24q^2+24}{24},\\
&&\vspace{0.01mm}\\
&& A_{q^2}^{(3)}=\frac{6q^9+9q^7+8q^6+7q^5+4q^4+14q^3+6q^2+6}{6},\\
&&\vspace{0.01mm}\\
&& A_{q^2}^{(4)}=\frac{ 2q^8 + 2q^7 + 2q^6 + 2q^5 + 5q^4 + 2q^3 + q^2 + 2q + 2}{2},\\
&&\vspace{0.01mm}\\
&& A_{q^2}^{(5)}=q^5 + q^4 + q^3 + q + 1,\\
&&\vspace{0.01mm}\\
&& A_{q^2}^{(6)}= 1, 
\end{eqnarray*}
and all other $A_w^{(r)}$ are zero. 













\bigskip
\item[{$(b)$}] For the  $[25,\, 6, \, 15]$-code $C_5$, we have: 
For $w \ne q^2-q+1=q^2-4=21$, all nonzero $A^{(r)}_w$ are given by the same formulas in $q$, as in {(a)}. Furthermore, $A^{(1)}_{21}=1500, $\ $A^{(2)}_{21}=6500$, and $A^{(r)}_{21}=0$ for $r=3,4,5,6.$

\bigskip
\item[{$(c)$}] For the  $[16,\, 6, \, 8]$-code $C_4$, we have:  
For  $w \ne q^2-q=q^2-4=12$ and  $w \ne q^2-q+1=q^2-3=13$, all nonzero $A^{(r)}_w$ are given by the same formulas in $q$, as in {(a)}. Furthermore, $A^{(1)}_{12}=280,$\ $A^{(2)}_{12}=1020,$\ $A^{(3)}_{12}=20$, and $A^{(r)}_{12}=0,$ for $r=4,5,6$ and $A^{(1)}_{13}=480,$\ $A^{(2)}_{13}=5280$,\ $A^{(3)}_{13}=480$, and $A^{(r)}_{13}=0$ for $r=4,5,6.$

\bigskip
\item[{$(d)$}] For  the $[9, \, 6, \, 3]$-code $C_3$, we have:
For $w \ne q^2-q-1=q^2-4=5$,\ $w \ne q^2-q=q^2-3=6$, and  $w \ne q^2-q+1=q^2-2=7$, all nonzero $A^{(r)}_w$ are given by the same formulas in $q$, as in {(a)}. Furthermore, $A^{(1)}_{5}=54,$\ $A^{(2)}_{5}=126$, and $A^{(r)}_{5}=0,$ for $r=3, 4,5,6$, and $A^{(1)}_{6}=96,$\ $A^{(2)}_{6}=588,$ \ $A^{(3)}_{6}=84$, and $A^{(r)}_{6}=0,$ for $r=4,5,6,$ and 
$A^{(1)}_{7}=108,$\ $A^{(2)}_{7}=2160,$\ $A^{(3)}_{7}=1188$, and $A^{(r)}_{7}=0$ for $r=4,5,6$.

\bigskip
\item[{$(e)$}] For  the $[4,  4,  1]$-code $C_2$, we have: $A_0^{(0)}=1$, 
$A^{(1)}_1=4$, $A^{(1)}_2=6,\ A^{(1)}_3=4,\ A^{(1)}_4=1$,  $A^{(2)}_2=6,\ A^{(2)}_3=16,\ A^{(2)}_4=13,\ A^{(3)}_3=4,\ A^{(3)}_4=11,\ A^{(4)}_4=1,$
and all other $A^{(r)}_w$ are zero. 
\end{itemize}
\end{theorem}

It may be noted that the formulas in parts (b), (c) and (d) of Theorem~\ref{Fallq} are consistent with those in part (a), in the sense that when we have two different expressions, say $w_1$ and $w_2$, for the weight $w$, then the actual value of $A_w^{(r)}$ is the sum of the formulas given in Theorem~\ref{Fallq}(a) for $A_{w_1}^{(r)}$ and $A_{w_2}^{(r)}$. For example, if $q=4$, then $q^2-q = q^2-4 =12$ and in this case, Theorem~\ref{Fallq}(a) gives  
$$
A_{q^2-q}^{(2)}=2q^3 + 3q^2 + q = 180 \quad \text{and} \quad 
A_{q^2-4}^{(2)}= \frac{q^8-4q^7+5q^6+q^5-6q^4+3q^3}{24} = 840,
$$
whereas Theorem~\ref{Fallq}(b) gives   $A_{12}^{(2)}= 1020$ and sure enough, $180 +840 = 1020$. With this in view, we obtain the following consequence of 
Theorem~\ref{Fallq}, which holds uniformly for $q\ge 3$.

\begin{corollary}\label{cor:UnifiedWtEnum}
Assume that $q\ge 3$ and $1\le r \le 6$. The $r$-th higher weight polynomial of $C_q  = \RM_q(2,2)$ is given by 
$$
  \sum_{w} A^{(r)}_w X^{q^2-w}Y^w
 $$
 where $w$ takes values in the (multi)set
 $$
 \{q^2 -2q, \; q^2 -2q+1, \; q^2 -q-1, \; q^2 -q, \; q^2 -q+1, \; q^2 -4, \; q^2 -3, \; q^2 -2, \; q^2 -1, \; q^2 -1, \; q^2 \}
 $$
 and $A_w^{(r)}$ is defined as in 
 Theorem~\ref{Fallq}(a). 
\end{corollary}


\section{Classification of Affine Conics and the First Weight Spectrum}
\label{sec3}

Our proof of  Theorem~\ref{Fallq} 
has the following 3 main steps. For each value of $q$:
\begin{itemize}
\item[(1)] 
Find the $\mathbb{N}$-graded minimal free resolutions of the Stanley-Reisner ring of the  matroid associated to a parity check matrix of $C_q$,  and  each of the $6$ elongations of this matroid. 
(These terms will be explained in Section \ref{matroids}.) 

\item[(2)] Use Step (1) and the procedure described in \cite{JRV} to find the so-called generalized weight polynomials $P_j(Z)$, for $j=0,\ldots,q^2,$ for each 
of the codes $C_q$. These polynomials $P_j$ are such that $P_j(q^m)$ gives the number of codewords of 
weight $j$ of
the extension code
$C_q \otimes _{\F_q}\F_{q^m}$. Concretely, we will calculate the polynomials $P_j(Z)$ by using the formula:
$$P_j(Z)=\sum_{\ell=0}^{6}(\phi_j^{(\ell)}-\phi_j^{(\ell-1)})Z^{\ell},
$$
where the $\phi_j^{(\ell)}$ are certain integers that we obtain from the $\mathbb{N}$-graded minimal free resolutions of the above-mentioned matroid and its elongations. The notation will be explained in detail in Section \ref{matroids}.
\item[(3)]  Use Step (2) to find the $A_{w}^{(i)}$, using the well known formula 
$$
P_w(q^e)= \sum_{r=0}^e
A_w^{(r)}\prod_{i=0}^{r-1}(q^e-q^i) \quad \text{for } e\ge 0,
$$
(given in \cite{HKM}, as pointed out in \cite{J}) repeatedly, starting with $e=0$.
\end{itemize}

The last two main steps are somewhat ``mechanical'' and  they only involve rather straightforward, but tedious computations.\footnote{A large computational part was done using the computer algebra software SageMath \cite{sage}. We have made the code used for the computations available at: \url{https://github.com/RakhiPra/BettiNumbers_HWS_ReedMullerCodes.git}
}
Hence we will devote most of our attention to Step (1). A vital ingredient there is the following result. Here, and hereafter, by a \emph{conic} in $\mathbb{P}^2_q$ we mean the set of zeros in $\mathbb{P}^2_q$ of a nonzero homogeneous polynomial  of degree $2$ in $3$ variables with coefficients in $\Fq$. 

\begin{proposition} \label{Hirsch}
In $\mathbb{P}^2_q$, for all prime powers $q \ge 2$,  the $\frac{q^6-1}{q-1}$ conics are as follows: 
\begin{enumerate}[itemsep=1ex]
\item[1.] 
$q^2+q+1$ double lines with $q+1$ $\Fq$-rational points,
\item[2.] 
$\frac{q(q+1)(q^2+q+1)}{2}$ pairs of two distinct lines with $2q+1$ $\Fq$-rational points, 
\item[3.] 
$q^5-q^2$ irreducible conics with $q+1$ $\Fq$-rational points, 
\item[4.] 
$\frac{q^4-q}{2}$ conics (pairs of Galois-conjugate lines defined over $\F_{q^2}\setminus \Fq$) that just possess a single $\Fq$-rational point each. 
\end{enumerate}
\end{proposition}
This result can be found, for instance, in \cite[p. 140]{Hi}. 

\begin{remark} \label{Hirschfeld}
{\rm The values of the $A_w^{(1)}$, for all $w$, and all prime powers $q \ge 3$, will be derived from Proposition \ref{Hirsch}, and the additional analysis below, where  for each of the $4$ conic types listed in Proposition \ref{Hirsch}, 
one finds out how they intersect the line $L$ at infinity given by  $z=0$.
 This is because the supports of the (one-dimensional subspaces of) codewords corresponding to the conics above are the complements in 
$\mathbb{A}^2_q= \mathbb{P}^2_q - L$ of the zero sets of the conics. More precisely, for $q \ge 3$ and $w\ge 0$, 
\begin{equation}\label{eq:conic-wt}
A_w^{(1)} = \left| \left\{ \mathsf{C} : \mathsf{C} \text{ conic in $\mathbb{P}^2_q$ with } |\mathsf{C} \cap \mathbb{A}^2_q| = q^2 - w\right\} \right|.
\end{equation}
}
\end{remark}


\begin{proposition} \label{conicvariations}
The $\frac{q^6-1}{q-1}$ conics in $\mathbb{P}^2_q$ can be further classified as follows: 
\begin{enumerate}[itemsep=1ex, parsep=1ex]
\item[1.] The $q^2+q+1$ double lines in $\mathbb{P}^2_q$ are divided into $2$ categories: 
\begin{enumerate}[itemsep=0.5ex]
\item[{$(a)$}] $1$ double line $z^2=0$, with no points in $\mathbb{A}^2_q.$
\item[{$(b)$}] $q^2+q$ other double lines, each with $q$ zeros in $\mathbb{A}^2_q.$
\end{enumerate}
\item[2.] The $\frac{q(q+1)(q^2+q+1)}{2}$ pairs of 
distinct lines in $\mathbb{P}^2_q$ are divided into $3$ categories:
\begin{enumerate}[itemsep=0.5ex]
\item[{$(c)$}] $q^2+q$ line pairs of the type $F(x,y,z)z=0$, where $F(x,y,z)$ is a linear form not proportional to $z$. 
So, one of the lines is $L$ and the affine part is 
the line given by $F(x,y,1)=0.$
These conics have $q$ zeros in $\mathbb{A}^2_q$.
\item[{$(d)$}] 
$\frac{q^4+q^3}{2}$ line pairs that intersect outside $L$. These conics have $2q-1$ zeros in $\mathbb{A}^2_q.$
\item[{$(e)$}] $\frac{q(q^2-1)}{2}$ line pairs, each of which intersect at a single point of the line $L$ at infinity. Such line pairs have $2q$ zeros in $\mathbb{A}^2_q$, where they appear as parallel lines. 
\end{enumerate}
\item[3.] The $q^5-q^2$ irreducible conics in $\mathbb{P}^2_q$ are divided into $3$ categories:
\begin{enumerate}[itemsep=0.5ex]
\item[{$(f)$}] $\frac{q^3(q^2-1)}{2}$ conics, each  intersecting $L$ in two distinct points over $\F_q$. The ``finite parts'' can be thought of as hyperbolas in $\mathbb{A}^2_q$. 
These conics have $q-1$ zeros in $\mathbb{A}^2_q.$
\item[{$(g)$}] $q^2(q^2-1)$ conics being tangent to $L$ at one $\F_q$ rational point. The ``finite parts" can be thought of as parabolas in $\mathbb{A}^2_q$. These conics have $q$ zeros in $\mathbb{A}^2_q.$
\item[{$(h)$}] $\frac{q^3(q-1)^2}{2}$ conics that have no  $\F_q$-rational point on $L$. It is natural to think of these conics as ellipses. These conics have $q+1$ zeros in $\mathbb{A}^2_q.$
\end{enumerate}
\item[4.] The $\frac{q^4-q}{2}$ 
conics in $\mathbb{P}^2_q$ that just possess a single $\F_q$-rational point each, are divided into 2 categories:
\begin{enumerate}[itemsep=0.5ex]
\item[{$(i)$}] $\frac{q^3(q-1)}{2}$ conics, where the single point is not on the line $L$ at infinity. These conics have $1$ point each in $\mathbb{A}^2_q.$
\item[{$(j)$}] 
$\frac{q^3-q}{2}$ conics, where the single point is  on the line $L$ at infinity. These conics have $0$ points each in $\mathbb{A}^2_q.$ 
\end{enumerate}
\end{enumerate}
\end{proposition}

\begin{proof}
The cases $(a)$--$(c)$ follow from straightforward counting.
Next, $(d)$ follows from Proposition~\ref{Hirsch} and from $(c)$ and $(e)$ since 
$$
\frac{q(q+1)(q^2+q+1)}{2} - (q^2+q) - \frac{q(q^2-1)}{2}=\frac{q^4+q^3}{2}.
$$
To prove (e), observe that there are $q+1$ points at infinity, and through each of them,  there are 
exactly $\binom{q}{2}$ pairs of distinct lines different from the line at infinity.
Hence there are $(q+1){\binom{q}{2}} = \frac{q(q^2-1)}{2}$ line pairs of the kind asserted in (e). 

In the case $(f)$, one looks at the incidence set 
$$
\left\{\left(\text{irreducible conic } \mathsf{C}\text{ in } \mathbb{P}^2_q,\text{ line in }\mathbb{P}^2_q \right) : \text{ the line intersects } \mathsf{C}\text{ in 2 points}\right\}.
$$ 
This incidence set clearly has $(q^5-q^2)\binom{q+1}{2}$ points. Projecting down on the right factor, one finds that there are 
$$
\frac{(q^5-q^2)\binom{q+1}{2}}{{q^2+q+1}}= \frac{q^5-q^3}{2}
$$ 
irreducible conics that intersect a fixed line in $2$ points. Since the line at infinity is no exception, we obtain $(f)$.

To prove $(g)$, one looks at a similar incidence set:
$$
\left\{\left(\text{irreducible conic } \mathsf{C}\text{ in } \mathbb{P}^2_q, \text{ line in }\mathbb{P}^2_q \right) : \text{ the line is tangent to  } \mathsf{C} \right\}.
$$
This incidence set clearly has $(q^5-q^2)(q+1)$ points. Projecting down on the right factor, one finds that there are 
$$
\frac{(q^5-q^2)(q+1)}{q^2+q+1}= q^4-q^2 
$$
irreducible conics that are tangent to a fixed line, in particular to the line $L$. 

To prove $(h)$, one takes the  number $q^5-q^2$ of all irreducible conics in $\mathbb{P}^2_q$, and subtracts from it  the numbers of conics from $(f)$ and $(g)$. 

Finally, to prove $(i)$ and $(j)$, we observe that among the $\frac{q^4-q}{2}$ conics in part 4 of Proposition~\ref{Hirsch} that just possess a single $\F_q$-rational point, the number, say $N$, of conics containing a given $P\in \mathbb{P}^2_q$ is 
independent of the choice of $P$. Since there are $q^2+q+1$ points in $\mathbb{P}^2_q$, we see that 
$N= \frac{q^4-q}{2(q^2+q+1)}$. Further, since there are $q^2$ points in $\mathbb{A}^2_q$ and $q+1$ points in $L$, it follows that the number of conics as in $(i)$ and $(j)$ are given by
$Nq^2 = \frac{q^3(q-1)}{2}$ and $N(q+1) = \frac{q^3-q}{2}$, respectively.
\end{proof}

\begin{corollary} \label{corconic}
{$(a)$} For $q \ge 3$, the first weight spectrum of $C_q$ is as follows:
\begin{eqnarray*}
\begin{aligned}[c]
 A_{q^2-2q}^{(1)}\quad & = \quad \frac{q(q^2-1)}{2},\\
 A_{q^2-2q+1}^{(1)}\quad & = \quad \frac{q^4+q^3}{2},\\
 A_{q^2-q-1}^{(1)}\quad & = \quad \frac{q^3(q-1)^2}{2},\\
  A_{q^2-q}^{(1)}\quad & = \quad q^4+q^2+2q,\\
\end{aligned}
\qquad
\begin{aligned}[c]
A_{q^2-q+1}^{(1)}\quad & = \quad \frac{q^3(q^2-1)}{2},\\
 A_{q^2-1}^{(1)}\quad & = \quad \frac{q^4-q^3}{2},\\
 A_{q^2}^{(1)}\quad &= \quad \frac{q^3-q+2}{2}, \\
 A_w^{(1)}\quad & = \quad 0, \text{ for all other }w.
\end{aligned}
\end{eqnarray*}









{$(b)$} For $q=2$, we have:  
$$
A_1^{(1)}=4, \ A_2^{(1)}=6, \ A_3^{(1)}=4, \ A_4^{(1)}=1, \text{ and $A_w^{(1)}=0$ for all other $w$.}
$$
\end{corollary}

\begin{proof} 
From Proposition \ref{conicvariations}, we see that the number of zeros in $\mathbb{P}^2_q$ of a conic  in $\mathbb{P}^2_q$ can be $0, \, 1 , \,  q-1, \, q, \, q+1, \, 2q-1$ or $2q$. When $q\ge 3$, these values are pairwise distinct and also different from $q^2$. This together with the relation \eqref{eq:conic-wt} readily yields the assertions in $(a)$. [For instance, parts 1(b), 2(c) and 3(g) of Proposition \ref{conicvariations} show that $A_{q^2-q}^{(1)}= (q^2+q) + (q^2+q) + q^2(q^2-1)=q^4+q^2+2q$.] 
When $q=2$, it suffices to note that $C_2=(\mathbb{F}_2)^4$, and moreover, for any positive integer $w$, 
$$
A_w^{(1)} 
= |\{ \mathbf{w} \in (\mathbb{F}_2)^4 : \wt({\bf w}) = w\}|.
$$
Thus the result in (b) follows from a straightforward counting. 
\end{proof}

\begin{remark} {\rm  If we had to use Proposition \ref{conicvariations} when $q=2$, then our definition of $C_2$ would require us to consider restricted class of conics given by nonzero homogeneous polynomials of degree $2$ in $3$ variables with no term in $xz$ or $yz$. Thus  we would need a variant of Proposition \ref{conicvariations} for such conics. But it is much simpler to argue directly by observing that  $C_2=(\mathbb{F}_2)^4$. 
}
\end{remark}

\begin{remark} \label{pardon}
{\rm The preceding proposition and corollary are included for completeness, and in order to give an overview for later use. A result equivalent to Corollary~\ref{corconic} is given in \cite[Proposition 2.12]{KM}.
Also, in \cite[Lemma 4.3/Theorem 4.4]{KM}, results are given that can be used to deduce our formulas for 
$A^{(2)}_w$. 
But as far as we know, there is no prior determination of general formulas for the $A^{(r)}_w$  for $r \ge 3$ and $q \ge 7.$
}
\end{remark}


We now embark on an analysis of codewords that have an inclusion-minimal support. This is equivalent to studying which of the conics listed above have inclusion-maximal zero sets in $\mathbb{A}_q^2$ over $\Fq.$ This analysis has no implication for the determination of the $A^{(1)}_w$, but the results will be useful for calculating the $A^{(r)}_w$, with $r \ge 2$, using the methods of Section~\ref{matroids}. 
We begin by recalling a version of Bezout's theorem (\cite[Theorem 2.1]{FRBZ}) for plane curves over arbitrary fields. 
\begin{theorem} \label{Bezout}
Suppose  $X$ and $Y$ are two plane projective curves defined over a field $F$, such that they do not have a common component (in other words, $X$ and $Y$ are defined by nonzero homogeneous polynomials in three variables with coefficients in $F$ such that these polynomials do not have a common divisor of positive degree). Then the total number of intersection points of $X$ and $Y$ with homogeneous coordinates in an algebraically closed field $K$ that contains $F$ is at most the product of the degrees of $X$ and $Y$.
\end{theorem}


\begin{proposition} \label{minwords}
For $q \ge 7$, the inclusion-maximal zero sets of conics in $\mathbb{A}_q^2$ over $\F_q$
 are precisely those listed in $(e), \, (d), \, (h), \, (g)$, and $(f)$ of Proposition~\ref{conicvariations}, with
 $2q, \, 2q-1, \, q+1, \, q$, and $q-1$ zeros each, respectively.
 \end{proposition}
\begin{proof}
 The zero sets of the conics in  Proposition \ref{conicvariations}$(a), (b), (c), (i)$, and $(j)$
 (namely, $\emptyset$, a double line, a line, a point, and $\emptyset$) are properly contained in zero sets of some other conics, e.g., those in Proposition \ref{conicvariations}$(e)$; so these are not inclusion-maximal. 
 The zero sets of type $(e)$ in  Proposition \ref{conicvariations} are cardinality-maximal and therefore inclusion-maximal.
 Moreover, for $q \ge 3$ (and in particular, for $q\ge 7$),   the zero sets of type $(d)$ in Proposition \ref{conicvariations} are also inclusion-maximal, since $q$ points on each line determine  that line uniquely, and a  pair of intersecting lines cannot then be contained in any other line pair, in particular not in one consisting of parallel lines, like in Proposition \ref{conicvariations}$(e)$.
 
 It remains to study the irreducible conics in  Proposition \ref{conicvariations}$(f),(g),(h)$. The most ``critical'' one is type $(f)$, which has only $q-1$ points in its zero set in $\mathbb{A}^2_q$. But for $q \ge 7$, we have $q-1 \ge 6$. These $6$ points cannot  be contained in a line pair. If it were so, then one of these lines must contain at least $3$ points. By Bezout's theorem the conic, taken from $(f)$, must then be reducible, which it isn't. Hence, the conics from  Proposition~\ref{conicvariations}$(f)$ are not contained in line pairs from  Proposition~\ref{conicvariations}$(d)$ and $(e)$. The same argument applies to the conics from  Proposition~\ref{conicvariations}$(g)$ and $(h)$;  so neither of them are contained in line pairs from  Proposition \ref{conicvariations}$(d)$ and $(e)$.
 Since the conics from  Proposition \ref{conicvariations}$(h)$ are cardinality-maximal among those 
conics that are not line pairs, it  is now clear that they are inclusion-maximal also. But it is also clear that the conics from  Proposition~\ref{conicvariations}$(g)$ cannot be contained in conics from $(h)$, and that the conics from  Proposition~\ref{conicvariations}$(f)$ cannot be contained in the conics from $(g)$ or $(h)$.
 This is because all irreducible conics, which we treat here, have at least $5$ (in fact $6$) $\F_q$-rational points in $\mathbb{A}^2_q$, and $5$ points on an irreducible conic always define the conic uniquely. Moreover, two of them intersect in at most $4$ points by Bezout's theorem. This gives the conclusion.
 \end{proof}
 
 The techniques of the proof above also give:
 \begin{corollary} \label{cor:smallq} 
 Suppose $q$ is a prime power with $q < 7$.
 \begin{enumerate}
     \item[$(a)$] For $q=5$, the inclusion-maximal zero sets of conics in $\mathbb{A}_q^2$ over $\F_q$
 are precisely those listed in  Proposition \ref{conicvariations}$(e), \, (d), \, (h)$, and $(g)$ with
 $2q, \, \hbox{$2q-1$}, \, \hbox{$q+1$}$, and $q$ zeros each, respectively.
 \item[$(b)$] 
 For $q=4$, the inclusion-maximal zero sets of conics in $\mathbb{A}_q^2$ over $\F_q$
 are precisely those listed in  Proposition~\ref{conicvariations}$(e), \, (d), \, (h)$ with
 $2q, \, 2q-1, \, q+1$ zeros each, respectively.
 \item[$(c)$] 
For $q=3$, the inclusion-maximal zero sets of conics in $\mathbb{A}_q^2$ over $\F_q$
 are precisely those listed in  Proposition \ref{conicvariations}$(e)$ and $(d)$,  with
 $2q$ and $2q-1$ zeros each, respectively.
 \item[$(d)$] 
 For $q=2$, the zero set of conics in  Proposition \ref{conicvariations}$(e)$ is all of $\mathbb{A}_q^2$, which then gives rise to the zero codeword, while the conics in  Proposition \ref{conicvariations}$(d), (h)$ give zero sets of cardinality $2q-1=q+1=n-1=3$,  which are maximal among those not giving the whole plane. Hence they give inclusion-minimal non-zero codewords. None of the other types give zero sets that are maximal among those not giving the whole plane.
  \end{enumerate}
 \end{corollary}
 
 \begin{proof}
 For $q=5$, the $q-1=4$ points in  Proposition \ref{conicvariations}$(f)$ are contained in a line pair.
  Also, for $q=4$,  the $q=4$ points  in  Proposition \ref{conicvariations}$(g)$ are contained in a line pair.
    Likewise, for $q=3$,  the $q+1=4$ points  in Proposition \ref{conicvariations}$(h)$ are contained in a line pair.
    For $q=2$, the assertions in $(d)$ are obvious. 
 \end{proof}




\section{Matroids and Betti Numbers} \label{matroids}

In this section, we shall review some basic facts about matroids, and highlight their connections with linear codes and Betti numbers of associated Stanley-Reisner rings.  
Much of this section is similar to the corresponding one in \cite{JV20}.
Throughout this section, we let $E$ denote a finite set. We begin with one of the several equivalent definitions of the basic notion of a matroid. 

\begin{definition}\label{def:matroid} 
{\rm A (finite) \emph{matroid} $M$ on the finite ground set $E$ is an ordered pair $(E,\mathcal{I}),$ where 
$\mathcal{I}$ is a collection of subsets of $E$ having the following properties:  
\begin{enumerate}
\item[$1.$] $\emptyset \in \mathcal{I}.$ 
\item[$2.$] If $I_1\in\mathcal{I}$ and $I_2\subseteq I_1,$ then $I_2\in\mathcal{I}.$ 
\item[$3.$] If $I_1,I_2\in \mathcal{I}$ and $|I_1|<|I_2|,$ then there exists an element $y \in I_2\setminus I_1$ such that $I_1 \cup \{y\} \in \mathcal{I}$.
\end{enumerate} }
\end{definition}
The subsets of $E$ that are elements of $\mathcal{I}$, are called \emph{independent sets} of the matroid $M$. A subset of $E$ that is not independent is called a \emph{dependent set}. 
An inclusion-maximal set in $\mathcal{I}$ is called a \emph{basis} of the matroid $M$. 
The third property in Definition~\ref{def:matroid} implies that any two bases of a matroid 
have the same cardinality. 

\begin{definition}\label{def:nullsetsetc}
{\rm 
Let $M = (E,\mathcal{I})$ be a matroid on $E$. 
\begin{itemize}
\item 
Given any subset $\sigma$ of $E$, the \emph{rank} of $\sigma$, denoted $r(\sigma)$, is the largest cardinality of an independent set contained in $\sigma$, while the \emph{nullity} of $\sigma$ 
is $n(\sigma)=|\sigma|-r(\sigma).$ We write $\rk(M) = r(E)$ and call it the \emph{rank} of $M$. It is not difficult to see that $0\le n(\sigma) \le |E|- \rk(M)$ for any $\sigma\subseteq E$. 
\item 
Given any subset $\sigma$ of $E$, if we let ${\mathcal{I}}_\sigma = \{ \tau \in \mathcal{I} : \tau \subseteq \sigma\}$, then $(\sigma,\, {\mathcal{I}}_\sigma)$ is a matroid, which is denoted by $M_\sigma$ and called the \emph{restriction} of $M$ to $\sigma$. 
\item 
A \emph{circuit} of $M$ is a minimal dependent set of $M$. A \emph{flat} of $M$ is a subset $\sigma \subseteq E$ such that $r(\sigma \cup \{y\}) > r(\sigma)$ for any $y \in E\setminus \sigma$. 
\item For each integer $i=0,\dots,|E|- \rk(M)$, we let 
$$
N_i = \{ \sigma : \sigma \subseteq E \text{ and } n(\sigma) = i\}
\quad \text{and} \quad
d_i(M) = \min\{|\sigma| : \sigma \in N_i\}.
$$
We may call $d_i(M)$ the $i$-th \emph{generalized Hamming weight} of $M$. 
\end{itemize}
}
\end{definition}

\begin{remark}\label{rem:newlynumbered}
{\rm
(i) The rank and nullity functions in a matroid $M = (E,\mathcal{I})$ satisfy the following basic property for any subsets $\sigma, \tau$ of $E$.
$$
r(\sigma \cup \tau) + r (\sigma \cap \tau) \le r (\sigma) + r(\tau)
\quad \text{and} \quad
n(\sigma \cup \tau) + n (\sigma \cap \tau) \ge n (\sigma) + n(\tau).
$$
In particular, taking $\tau = E\setminus \sigma$, we obtain $0\le n(\sigma) \le |E|-\rk(M)$ for any $\sigma \subseteq E$.
(ii) It is clear from Definition~\ref{def:nullsetsetc} that the circuits of a matroid $M$ are precisely the inclusion-minimal elements of $N_1$. Thus the two terms may be used interchangeably. 
}
\end{remark}
\begin{definition}
{\rm 
Let $M = (E,\mathcal{I})$ be a matroid on $E$. If we let $\mathcal{I}^*$ denote the collection of subsets $J$ of $E$ such that $J\cap B=\emptyset$ for some basis $B$ of $M$, then $M^*=(E, \mathcal{I}^*)$ is a matroid of rank $|E|-\rk(M)$ and it is called the \emph{dual} of $M$. It can be shown that the rank function of $M^*$, denoted by $r^*$, is given by 
\begin{equation}\label{eq:rrstar}
r^*(\sigma) = r(E\setminus \sigma) +|\sigma| - r(E) \quad \text{for any } \sigma \subseteq E.   
\end{equation}
}
\end{definition}

\begin{ex}
{\rm
Given any positive integers $s, n$ with $s\le n$, let $U(s,n)$ be the corresponding \emph{uniform matroid}, namely,  $(E, \mathcal{I})$, where $E=\{1, \dots , n\}$ and $\mathcal{I}$ consists of all subsets $\sigma$ of $E$ such that $|\sigma|\le s$. Clearly, $U(s,n) = (E, \mathcal{I})$ is a  matroid with $\rk\left( U(s,n)\right) = s$ and $d_i\left( U(s,n)\right) = s+i$ for $1 \le i \le n-s$. Also, it is easy to see that the dual of $U(s,n)$ is $U(n-s,n)$.
}
\end{ex}

\begin{ex}
{\rm
Let $C$ be a linear $[n,k]_q$-code and let $H$ be a parity check matrix of $C$. Suppose $M$ is the \emph{vector matroid} corresponding to $H$, i.e., $M=(E, \mathcal{I})$, where $E=\{1, \dots , n\}$ and $\mathcal{I}$ consists of subsets $\sigma$ of $E$ such that the columns of $H$ indexed by $\sigma$  are linearly independent. The matroid $M$ is independent of the choice of a parity check matrix of $C$, and we may denote it by $M_C$ and call it the  \emph{parity check matroid} of $C$.  Note that  $M_C$ is a matroid with  $\rk (M_C) = n-k$ and $d_i(C) = d_i(M_C)$ for $i=1, \dots , k$ (see e.g.~\cite{JV}). It is easy to see that the dual of $M_C$ is the vector matroid corresponding to a generator matrix of $C$.
 }
\end{ex}

One way to construct new matroids from a given matroid is to consider its elongations, which are defined as follows. We will see  later that these elongations are an important tool in order to find the higher weight spectra.

\begin{definition} \label{elong}
{\rm 
Let $M =(E,\mathcal{I})$ be a matroid on $E$ and let $\ell$ be a nonnegative integer. Define $M^{(\ell)} = \left(E, \mathcal{I}^{(\ell)}\right)$, where 
$$
\mathcal{I}^{(\ell)} = \left\{ I \cup \sigma : I \in \mathcal{I} \text{ and $\sigma \subseteq E$ with } |\sigma| \le \ell \right\}.
$$
Then  $M^{(\ell)}$ is a matroid on $E$ and it is called the \emph{$\ell$-th elongation matroid} of $M$. 
}
\end{definition}

We note that if $M$ and $\ell$ are as in Definition~\ref{elong}, then $M^{(0)} = M$, whereas if $\ell \ge |E| -\rk (M)$, then  $M^{(\ell)}$ is the trivial matroid on $E$ in which every subset of $E$ is independent. Also,  the rank and nullity functions $r^{(\ell)}$, $n^{(\ell)}$ on $M^{(\ell)}$ are given~by 
$$
r^{(\ell)}(\sigma) = \begin{cases}
    r(\sigma) + \ell  & \text{ if } r(\sigma) + \ell \le |\sigma|, \\
    |\sigma| & \text{ otherwise}
    \end{cases}
\quad \text{and} \quad
   n^{(\ell)}(\sigma) = \begin{cases}
    n(\sigma) - \ell  & \text{ if } n(\sigma) \ge  \ell, \\
    0 & \text{ otherwise.} 
\end{cases}
$$
In particular, $\rk \left(M^{(\ell)} \right)  =\min (|E|, \; \ell+\rk(M)) $ and $\mathcal{I}^{(\ell)} =\{ \sigma \subseteq E : n(\sigma) \le \ell\}$.


\begin{ex}\label{ex:unifelong}
{\rm If 
$M= U(s,n) $ and $0\le \ell \le n -s$, then $M^{(\ell)} = U(s+\ell, \, n) $.
}
\end{ex}
We now fix the notation and terminology for the matroid associated to the code that we are interested in. 

\begin{definition}
{\rm 
Let 
$E_q=\{1,2,\dots,q^2\}$ and let $M_q=(E_q,\mathcal{I}_q)$ be the parity check matroid of the Reed-Muller code  $C_q=\RM_q(2,2)$.  We may  identify $E_q$ with $\mathbb{A}^2_q$, for a fixed ordering of the points of the affine plane.
}
\end{definition}


Next, we review some basic facts about (abstract) 
simplicial complexes,  their Stanley-Reisner rings and their minimal free resolutions and Betti numbers. We will also state some useful properties of simplicial complexes arising from matroids. 


Let $\Delta$ be a simplicial complex on  $E$ (which means $\Delta$ is a collection of subsets of $E$ such that if $\sigma \in \Delta$ and $\tau \subseteq \sigma$, then $\tau\in \Delta$). The elements of $\Delta$ are called \emph{faces}, and maximal elements under inclusion are called \emph{facets}.
If for each nonnegative integer $i$, we let $f_i$ denote the number of faces of $\Delta$ of cardinality $i$, then the alternating sum 
$\sum_{i\ge 0} (-1)^{i+1} f_i$ is called the \emph{Euler characteristic} of $\Delta$ and denoted by $\chi(\Delta)$. 
If $E$ is nonempty, then one may use the alternative expression  
$\chi(\Delta) = \sum_{i\ge 0} (-1)^i D_i$, 
where $D_i$ is the number of dependent sets of cardinality $i$ contained in $E$, 
i.e., $D_i=|\{\tau\subseteq E: |\tau|=i \text{ and } \tau \notin \Delta\}|$. This is because if $E\ne \emptyset$, then $|E|\ge 1$ and 
$$
\sum_{i\ge 0} (-1)^i D_i - \sum_{i\ge 0} (-1)^{i+1} f_i = \sum_{i\ge 0} (-1)^{i} (D_i+f_i) 
= \sum_{i\ge 0} (-1)^{i} \binom{|E|}{i} = (1 -1)^{|E|} = 0. 
$$
Note that if $M =(E,\mathcal{I})$ is a matroid on $E$, then the collection $\mathcal{I}$ of independent sets for $M$ forms a simplicial complex, called a \emph{matroid complex}; the Euler characteristic of this simplicial complex may be denoted by $\chi(M)$ rather than  $\chi(\mathcal{I})$.

Now let us fix a field $\k$ and let $S$ be the polynomial ring over $\k$ in $|E|$ indeterminates indexed by $E$, say $\{X_e : e\in E\}$.  We denote by $\N^E$ 
  the set of all maps from $E$ into the set $\N$ of nonnegative integers.  
  Elements of $\N^E$ can be identified with the monomials in $S$, and for $\nu \in \N^E$, we denote by $\mathbf{x}^{\nu}$ the corresponding monomial, viz., $\prod_{e\in E} X_e^{\nu(e)}$, and by $|\nu|$ the degree of this monomial, viz., $\sum_{e\in E} \nu (e)$. The polynomial ring $S$ admits a natural $\N^E$-grading and an $\N$-grading given as follows:
  $$
  S = \bigoplus_{\nu \in \N^E}  S_{\nu}  \quad \text{and} \quad
  S = \bigoplus_{j\in \N} S_j, \quad \text{where} \quad S_{\nu} :=\k\mathbf{x}^{\nu} \quad \text{and} \quad  S_j := \bigoplus_{|\nu|=j} S_{\nu}.
  $$
  Note that for any $j\in \N$, the set $S_j$ consists of all homogeneous polynomials in $S$ of degree $j$ (including the zero polynomial). For any $\mu \in \N^E$ and any $m\in \N$, we denote by $S(-\mu)$ and $S(-m)$ the ring $S$ with the shifted grading given by 
  $$
  S(-\mu)_{\nu} = S_{\mu+\nu} \quad \text{and} \quad S(-m)_j = S_{m+j} \quad \text{for }  \nu \in \N^E \text{ and } j\in \N.
  $$
  The set $\{0,1\}^E$ of all maps from $E$ into $\{0,1\}$ is a subset of $\N^E$ and elements of $\{0,1\}^E$ can (and will) be identified with subsets of $E$. These correspond to squarefree monomials: given any $\sigma \subseteq E$, the corresponding monomial $\mathbf{x}^\sigma$ is $\prod_{e\in \sigma} X_e$. 
The \emph{Stanley-Reisner ideal} $I_\Delta$ of the simplicial complex $\Delta$ is the ideal of $S$ generated by the monomials corresponding to non-faces, and the   \emph{Stanley-Reisner ring} $R_\Delta$ of $\Delta$ is the corresponding residue class ring, i.e.,  
  \[
  I_\Delta = \left<\mathbf{x}^\sigma : \sigma \subseteq E \text{ and } \sigma \not \in \Delta\right> \quad 
  \text{and} \quad 
R_\Delta = S/I_\Delta
  \]
Clearly, $I_\Delta$ is a monomial ideal and thus $R_\Delta$ is $\N^E$-graded and it admits a minimal free resolution of the form 
\begin{equation} \label{SR}
0 \longleftarrow R_\Delta  \overset{\partial_0}{\longleftarrow} F_0  \overset{\partial_1}{\longleftarrow} F_1 \longleftarrow \cdots \overset{\partial_k}{\longleftarrow} F_k\longleftarrow 0 
\end{equation}
where each $F_i$ is a nonzero $\N^E$-graded free $S$-module of the form 
\[
F_i = \bigoplus_{\mu \in \N^{E}}S(-\mu)^{\beta_{i,\mu}}.
\] 
Here $\beta_{i,\mu}$ are nonnegative integers that are  independent of the choice of the minimal free resolution and they are called the \emph{$\N^{E}$-graded Betti numbers} of $\Delta$.
Moreover,  since $I_\Delta$ has square-free generators, the components $S(-\mu)^{\beta_{i,\mu}}$ of $F_i$ corresponding to $\mu \in \N^E\setminus \{0,1\}^E$ are all zero. (See, e.g., \cite[Corollary 1.40]{MS}.) Thus, we can, in fact write 
\[
F_i = \bigoplus_{\alpha \subseteq {E}}S(-\alpha)^{\beta_{i,\alpha}}\quad \text{and also} \quad 
F_i = \bigoplus_{j\in \N} S(-j)^{\beta_{i,j}}, 
\] 
where for $0\le i \le k$ and for any nonnegative integer  $j$, we have set 
\begin{equation} \label{eq:VariousBeta}
\beta_{i,j} = \sum_{|\alpha| = j}\beta_{i,\alpha}.
\end{equation}
We call $\beta_{i,j}$ 
the \emph{$\N$-graded Betti numbers} 
of $\Delta$.   
The minimal free resolution \eqref{SR} associated to $R_\Delta$ is said to be \emph{pure} of type $(d_0, d_1, \ldots, d_k)$ if for each $i=0,1, \dots , k$, the Betti number $\beta_{i,j}$ is nonzero if and only if $j=d_i$. If the resolution \eqref{SR} is pure and in addition, $d_1, \dots , d_k$ are consecutive, then it is said to be \emph{linear}. 

In case the simplicial complex $\Delta$ is a matroid complex given by the independent sets for a matroid $M$, then the Betti numbers $\beta_{i,\alpha}$ and $\beta_{i,j}$ are also independent of the choice of base field $\k$ (see, e.g., \cite[Remark 1]{JV}), and we call these the 
Betti numbers of $M$ (or of the code $C$ in case  $M$ is the parity check matroid $M_C$ of a linear code $C$). Moreover, the length $k$ of the minimal free resolution is  precisely  the rank of the dual matroid $M^*$ (and the dimension of $C$,  in case $M=M_C$ for a code $C$). The $\N^E$-graded and $\N$-graded Betti numbers of the $\ell$-th elongation $M^{(\ell)}$ of $M$ will be denoted by 
$\beta_{i,\alpha}^{(\ell)}$ and $\beta_{i,j}^{(\ell)}$, respectively. 

For a more detailed explanation of the concepts described in the preceding few paragraphs, see  
\cite[Pages 265-269]{HH} and \cite[Section 1.5]{MS}.

\begin{example} \label{bettiremark}
{\rm As mentioned in the introduction, Betti numbers are finer invariants for a linear Hamming metric code than its higher weight spectra.
In this paper we have determined all $\mathbb{N}$-graded Betti numbers of (all elongations of) the parity check matroid of our chosen code.   
We now give an example of two codes with same higher weight spectra, but different Betti numbers. The example is essentially taken from \cite[Example 6.1]{JRV}.
Let $M$ and $N$ be the vector matroids of the following two matrices over $\mathbb{F}_5$, respectively (they can thus be viewed as parity check matrices of two $[7,4,2]_5$-codes):

\[\begin{pmatrix}
    1 &0 &0 &3 &3 &3 &4\\
0 &1 &0 &0 &2 &2 &0\\
0 &0 &1 &4 &4 &4 &4
\end{pmatrix} \quad \text{and} \quad \begin{pmatrix}
    1 &0 &0 &1 &0 &1 &1\\
0 &1 &0 &1 &1 &1 &0\\
0 &0 &1 &0 &1 &0 &-1
\end{pmatrix}.
\]

As shown in \cite[Example 6.1]{JRV}  $M$ and $N$ have the same generalized weight polynomials, but different Betti numbers. The reason is they have different lattices of cycles. For example, $M$ has six circuits of size 4 (\cite[Example 2.3]{JRV}), while $N$ has only 5 circuits of size 4 
.}
\end{example}

The following two results, taken from \cite{JV}, will be useful:

\begin{theorem} \label{prep}
Let $M$ be a matroid on  $E$, and let $\Delta=\mathcal{I}$ be the simplicial complex on $E$ formed by the set of independent sets of $M$. 
\begin{enumerate}
\item[$(a)$]
Let $\sigma \subseteq E$. Then the corresponding $\N^E$-graded Betti numbers of $\Delta$ satisfy:
\[
\beta_{i,\sigma} \neq 0 \Longleftrightarrow \sigma 
\text{ is inclusion-minimal in }N_{i} \quad \text{(for any $i\ge 0$).}
\] 
\item[$(b)$] 
If $i$ is a nonnegative integer and $\sigma \in N_{i}$, then $\beta_{i,\sigma} = (-1)^{r(\sigma)-1}\chi(M_\sigma)$. 
\item[$(c)$] 
If $M$ is the matroid of a parity check matrix of a linear $[n,k]_q$-code $C$, 
then $d_i(C)= \min \{j : \beta_{i,\sigma}\ne 0$ for some $\sigma\subseteq E$ of cardinality $j\}$ for  $i= 1, \dots , k$.
\end{enumerate}
\end{theorem}

\begin{corollary} \label{betainfo}
Let $M$ be a matroid on  $E$.  
\begin{enumerate}
\item[$(a)$]
Let $\sigma \subseteq E$. Then
$$
\beta_{0,\sigma} = \left\{\begin{array}{ll} 1 & \text{ if } \sigma = \emptyset, \\ 0 & \text{ otherwise }\end{array}\right.
\quad \text{and } \quad
\beta_{1,\sigma} =  \left\{\begin{array}{ll} 1 & \text{ if } \sigma \text{ is a circuit,} \\ 0 & \text{ otherwise.}\end{array}\right.
$$
\item[$(b)$] 
The resolution has length exactly $k=|E|-\rk(M)$, that is: 
$N_k \ne \emptyset$, but $N_i =\emptyset$ for $i >k.$
\end{enumerate}
\end{corollary}

We also note a useful consequence of the above results for pure resolutions. 

\begin{corollary} \label{pureresol}
Suppose $|E|=n$ and $M$ is a matroid on  $E$ of rank $n-k$. Let $R_\Delta$ be the Stanley-Reisner ring of  the simplicial complex of independent sets in $M$.   
\begin{enumerate}
\item[$(a)$]
A minimal free resolution of $R_\Delta$ is pure of type $(d_0, d_1, \dots , d_k)$ if and only if every minimal element of $N_i$ has cardinality $d_i$, for $i=0,1, \dots , k$. 
\item[$(b)$] 
If a minimal free resolution of $R_\Delta$ is pure of type $(d_0, d_1, \dots , d_k)$, then for $0\le \ell \le k$, every minimal free resolution of the Stanley-Reisner ring of the simplicial complex of independent sets in $M^{(\ell)}$ is pure of type $(0, d_{\ell+1}, \dots, d_k)$.
\end{enumerate}
\end{corollary}

\begin{proof}
    From \eqref{eq:VariousBeta} and part (a) of Theorem~\ref{prep}, we see that
    \begin{eqnarray*}
    & &     \text{a minimal free resolution of $R_\Delta$ is pure of type $(d_0, d_1, \dots , d_k)$} \\
        &\Longleftrightarrow& |\sigma| = d_i \text{ whenever $\sigma \subseteq E$ is such that $\beta_{i, \sigma} \ne 0$ (for any $i\ge 0$)} \\
         &\Longleftrightarrow&  \text{ every minimal element of $N_i$ has cardinality $d_i$ (for any $i\ge 0$)}.
    \end{eqnarray*}
    This proves  $(a)$. Note that if the equivalent conditions in  $(a)$ hold, then $d_0=0$. Now  $(b)$ follows from  $(a)$ by observing that for any $0\le \ell \le k$ and $1\le i \le k- \ell$, the set $N^{(\ell)}_i $ of subsets of $E$ whose nullity with respect to the $\ell$-th elongation matroid $M^{(\ell)}$ is $i$, is  precisely $N_{i+\ell}.$ 
    This observation follows directly from Definition \ref{elong}. 
\end{proof}

We now recall a useful variant of a fundamental result from commutative algebra. 

\begin{proposition} \label{BS-HK}
The $\N$-graded Betti numbers $\beta_{i,j}$ of the Stanley-Reisner ring $R_{\Delta}=S/I_{\Delta}$ of a simplicial complex $\Delta$ associated to a matroid $M$ of rank $n-k$ on a ground set $E$ with $n$ elements   satisfy the Boij-S\"oderberg equations: 
\begin{equation}\label{eq:BSoriginal}
\sum_{i=0}^k\sum_{j=0}^n(-1)^ij^s\beta_{i,j}=0,\quad \text{for } s=0,1,\ldots,k-1.
\end{equation}
In particular, if the minimal free resolutions of $R_\Delta$ are pure of type $(d_0,d_1,\ldots,d_k)$, then \eqref{eq:BSoriginal} implies the Herzog-K\"uhl formula: 
\begin{equation}\label{eq:HKeqn}
\beta_{i,d_i}=\beta_{0,d_0}\left\vert\prod\limits_{\substack{1\le j\le k\\ j \neq i}} \frac{d_j}{(d_j-{d_i})}\right\vert \quad \text{for } i= 1,\ldots,k.
\end{equation}

\end{proposition}
\begin{proof} From \cite{BS}, it follows that the equations~\eqref{eq:BSoriginal} hold if and only if 
$R_\Delta$ 
is Cohen-Macaulay. But 
it is well known that a matroid complex 
is shellable, and therefore the corresponding Stanley-Reisner ring $R_\Delta$ 
is Cohen-Macaulay. The formula in \eqref{eq:HKeqn} follows from \eqref{eq:BSoriginal} using elementary linear algebra (see, e.g., \cite{HK}). 
\end{proof}

Here is a coarser form of $\N$-graded Betti numbers that will be  useful for us. 

\begin{definition} \label{phis}
{\rm 
Let $M, \Delta$ and $\beta_{i,j}$ as in Proposition~\ref{BS-HK}.  Define $\phi_j = \phi_j(\Delta)$ by 
$$
\phi_j=\displaystyle\sum_{i=0}^k(-1)^i\beta_{i,j} \quad \text{for } j=0, 1, \dots, n. 
$$
Further, for $0\le \ell \le k$, we denote the simplicial complex associated to the $\ell$-th elongation $M^{(\ell)}$ of $M$ by $\Delta_\ell$, and its $\N$-graded Betti numbers by $\beta_{i,j}^{(\ell)}$. Also, define 
$$
\phi_j^{(\ell)} = \displaystyle\sum_{i=0}^{k-\ell}(-1)^i\beta_{i,j}^{(\ell)} \quad \text{for } j=0, 1, \dots, n. 
$$

 }
\end{definition}

We get a neater version of the equations in Proposition \ref{BS-HK} if we express them with the $\phi_j$. 
In other words, equations \eqref{eq:BSoriginal} can be rewritten as follows. 

\begin{corollary} \label{HK-BS}
If $\Delta$ is as in Proposition~\ref{BS-HK} and $\phi_j = \phi_j(\Delta)$ for $0\le j \le n$, then 
\begin{equation}\label{eq:BS-phi}
\sum_{j=0}^nj^s\phi_{j}=0 \quad \text{for } s=0,1,\ldots,k-1.
\end{equation}
\end{corollary}

We may refer to the equations \eqref{eq:BSoriginal} as well as \eqref{eq:BS-phi} as the \emph{Boij-S\"oderberg equations}. It may be noted that in these equations, we adopt the convention that $0^0=1$. 

\begin{remark} \label{newphi}
{\rm 
The $\phi_j$, as opposed to the $\beta_{i,j}$ in general, are determined by the Hilbert series of $R_{\Delta}$. See \cite[Page 119]{E} or \cite[Proposition 6]{R}. 
To obtain the generalized weight polynomials and the higher weight spectra of the code $C_q$, we only need to know the $\phi_j$, and not necessarily the $\beta_{i,j}$, for $M_q$ and its elongations. 
This can be an advantage sometimes, not least 
because  \eqref{eq:BS-phi} gives us
linearly independent equations (regarding $\phi_j$ as variables), which is not necessarily the case with the equations  \eqref{eq:BSoriginal} (where one regards $\beta_{i,j}$ as variables). }
\end{remark}


Here are some more basic notions related to matroids and a useful relation with Betti numbers. 

\begin{definition}\label{recur}
{\rm 
Let $M$ be a matroid on  $E$. 
\begin{itemize}
    \item 
    A \emph{cycle} of $M$ is a subset of $E$ which is inclusion minimal in $N_i$ for some $i \ge 0$. The set of all cycles of $M$ forms a lattice (with respect to the inclusion relation) and we denote it by $L_M$.
    \item 
    The \emph{M\"obius function} of $M$, denoted $\mu_M^{ }$, is an integer valued function on $L_M$, which is defined recursively as follows. 
    $$
    \mu_M^{ }(\emptyset)=1 \quad \text{and for $\sigma \in L_M$ with $\sigma \ne \emptyset$,} \quad 
    \mu_M^{ } (\sigma) = - \sum_{\tau} \mu_M^{ } (\tau),
    $$
    where the sum is taken over all $\tau\in L_M$ such that $\tau \subsetneq \sigma$. 
\end{itemize}
}
\end{definition}


We remark that the lattice $L_M$ of cycles in a matroid $M$ is opposite of the lattice of flats of the dual matroid $M^*$. The following relation with Betti numbers is well known. See, for example, 
\cite{Stanley1977} or \cite{JPR}.

\begin{proposition} \label{mu}
Let $M$ be a matroid on  $E$. Then for $0\le i \le |E|-\rk(M)$, 
$$
\beta_{i,\sigma}=|\mu_M^{ }(\sigma)| \quad \text{for any inclusion minimal element $\sigma$ of $N_i$}.
$$
In particular, $\mu_M^{ }(\sigma)=-1$ and $\beta_{1,\sigma}=1$ for every circuit $\sigma$ in $M$.  
\end{proposition}

With the technical tools above, we are well equipped to take up the first part of main step 1, sketched in the beginning of Section~\ref{sec3}. 
This will be accomplished in the next three sections. 
\section{Identifying the Minimal Sets of a Given Nullity}\label{sec5}

In this section, we consider the parity check matroid 
$M_q=(E_q, \mathcal{I}_q)$ of  the Reed-Muller code $C_q = \RM_q(2,2)$, where $E_q=\{1, \dots , q^2\}$. We may denote $E_q$ simply by $E$ in this section, and we may tacitly identify $E$ with $\mathbb{A}_q^2$. We will denote by $r$ and $n$ the rank function and the nullity function of $M_q$, respectively. The sets $N_i$, as  in Definition \ref{def:nullsetsetc} will be the ones corresponding to the matroid $M_q$. Our aim is to find an  $\mathbb{N}$-graded minimal free resolution of the Stanley-Reisner ring associated to $\mathcal{I}_q$. To this end, it is clear from Theorem \ref{prep} that it would be useful to find the inclusion-minimal sets of $N_i$, for $i=1,\dots,6$,  since the local contributions $\beta_{i,\sigma}$ to the Betti numbers $\beta_{i,j}$ come from inclusion-minimal members of $N_i$, of cardinality $j$. 
For simplicity, we may drop the prefix ``inclusion" and simply write \emph{minimal} to mean inclusion-minimal. 
A crucial result for identifying members of $N_i$ in general, and the minimal ones in particular, is the following.


\begin{lemma} \label{help}
Assume that $q\ge 3$. 
Then the nullity of any $\sigma \subseteq E$  is equal to the $\Fq$-vector space dimension of the space of affine conics in $\mathbb{A}_q^2$ passing through $E\setminus \sigma$. More precisely, 
$$
n(\sigma) = \dim_{\Fq} \left\{f \in \Fq[x,y]_{\le 2} : f(P) = 0 \text{ for all } P\in E\setminus \sigma\right\},
$$
where $\Fq[x,y]_{\le 2}$ denotes the space of polynomials in two variables of degree at most 2 with coefficients in $\Fq$. 
\end{lemma}

\begin{proof}
The dual matroid of $M_q$ is the vector matroid corresponding to a generator matrix, say $G_q$ of $C_q$. 
Let $r^*$ denotes the rank function of this dual matroid and let $\sigma \subseteq E$. Then in view of \eqref{eq:rrstar}, we see that  
$$
n(\sigma) = |\sigma| - r(\sigma) = r^*(E) - r^*(E\setminus \sigma) = \mathrm{rank}(G_q) - \dim_{\Fq} \langle G_q^j : j \in E\setminus \sigma\rangle,
$$
where $G_q^j$ denotes the $j$-th column of $G_q$. Since $q\ge 3$, we must have  
$$
\mathrm{rank}(G_q) = \dim C_q = 6 = \dim_{\Fq} \Fq[x,y]_{\le 2}. 
$$
Also, the space $\langle G_q^j : j \in E\setminus \sigma\rangle$ spanned by the columns of $G_q$ indexed by the elements of $E\setminus \sigma$ can be viewed as the image of the projection map, say $\pi$,  that sends a codeword $(c_j)_{j\in E}$ of $C_q$ to $(c_j)_{j\in E\setminus \sigma}$. Hence the Rank-Nullity Theorem  shows that $n(\sigma)$ is the dimension of the kernel of $\pi$. 
Since $q\ge 3$, each codeword $c$ of $C_q$ corresponds to a unique $f\in \Fq[x,y]_{\le 2}$ with $c=(f(P))_{P\in E}$. So the kernel of $\pi$ can be  identified with the space 
$\left\{f \in \Fq[x,y]_{\le 2} : f(P) = 0 \text{ for all } P\in E\setminus \sigma\right\}$. This proves the lemma. 
%
\end{proof}

We are now ready to give our important theorem which categorizes and counts the minimal elements of $N_i$ for all $1\le i\le 6$. 
\begin{theorem} \label{nullity} 
The matroid $M_q$ satisfies: 
\begin{itemize}
\item[$(a)$] For $q \ge 7$, the minimal elements of $N_1$ are the complements of those listed in Proposition~\ref{minwords}.  For $2 \le q \le 5,$  the minimal subsets of $N_1$ are the complements of those listed in Corollary~\ref{cor:smallq}. For all $q \ge 2$, these sets are described in detail in Proposition~\ref{conicvariations}.

\item[$(b)$] Assume that $q \ge 3$. Then we have the following. 
\begin{itemize}
    \item 
The minimal elements of $N_2$ are all the $(q^2+q)(q^2-q)$ complements of  
$\{ \text{all }q \text{ points on a line}\} \cup \{\text{an additional point}\}$,  and all the ${q^2(q^2-1)(q^2-q)(q^2-3q+3)}/{24}$ complements of  ``qudrilaterals'' (four points such that no three of them are collinear).
\item
The minimal elements of $N_3$ are all the ${q^2+q}$ complements of $q$ points on a line, and all 
the ${q^2(q^2-1)(q^2-q)}/{6}$ complements of ``triangles'' (three non-collinear points).
\item
The minimal elements of $N_4$ are all the ${q^2 \choose 2}$ complements of point pairs.
\item
The minimal elements of $N_5$ are all the $q^2$ complements of a single point.
\item 
The only (minimal) element of $N_6$ is $E$.
\end{itemize}
\item[$(c)$] For $q=2$, the inclusion minimal elements of $N_i$, for $i=0,1,\ldots,4$ are all the subsets of $(\F_2)^2$ of cardinality $i$.
\end{itemize}
\end{theorem}
\begin{proof}
$(a)$ 
For all $q \ge 2$, the minimal elements of $N_1$ (also called circuits) are the complements of sets that are maximal among the supports of conics in $\A^2_q$. 
Thus the statement follows from  Proposition~\ref{minwords} and Corollary~\ref{cor:smallq}.

$(b)$ Assume that $q \ge 4$. We first determine minimal sets of $N_2$. 
Any set 
$\sigma$ of cardinality at most  $q^2-q-2$ is the complement of a set of cardinality at least $q+2$, which, in case it lies on a conic, must lie on a line pair, which it defines uniquely. Hence the nullity of $\sigma$ is $0$ or $1$. 

Let $E\setminus \sigma$ of cardinality $q+1$ be such that it lies on a conic (a necessary condition for  $\sigma$ to have positive nullity, by Lemma~\ref{help}). If it lies on an irreducible conic, it cannot lie on any other conic by Bezout's theorem (Theorem~\ref{Bezout}),
since 
$q+1 \ge 5 > 2 \times 2=4.$
Thus $n(\sigma)=1$, by Lemma~\ref{help}.  Hence a necessary condition for $\sigma$ to have nullity $2$ is that $E\setminus \sigma$ lies on a line pair, which it does not define uniquely. If there are at least $2$ points from $E\setminus \sigma$ on each line in such a pair, then the  lines are indeed defined uniquely. Hence the only possibility for $\sigma$ to have nullity at least $2$ is that $E\setminus \sigma$ consists of all $q$ points on the line, and $1$ point outside the line.
In that case it indeed has nullity exactly $2$.
Moreover, if we add one point to $E\setminus \sigma$, then the other line of a line pair  containing the extended set, is also defined uniquely, so the nullity of its complement, strictly contained in $\sigma$ is $1$. Hence $\sigma$ is minimal in $N_2$. 
There are clearly $(q^2+q)(q^2-q)$ choices of such pairs of a line and a point outside the line, and therefore of such minimal $\sigma$ of cardinality $q^2-q-1$.

Now consider any $E\setminus \sigma$, such that its cardinality satisfies $5 \le |E\setminus \sigma| \le q$ (or $q^2-q \le |\sigma| \le q^2-5$ if one prefers), and such that $E\setminus \sigma$ lies on a conic. 
If $E\setminus \sigma$ lies on an irreducible conic, then $n(\sigma)=1$, since $E\setminus \sigma$ cannot lie on two conics, by Bezout's theorem. If $E\setminus \sigma$ lies on a line, then $n(\sigma)=3$ by Lemma~\ref{help}, a case which will be studied later.
If $E\setminus \sigma$ lies on a line pair, then $n(\sigma)=1$ if the line pair is defined uniquely, so this case can be excluded. This implies that all but one of the points in $E\setminus \sigma$ are on one of the lines of the line pair, while the last point is outside it. But then, since $|E\setminus \sigma| < q+1$, this $E\setminus \sigma$ is strictly contained in the set of $\it{all}$ points on the line, and the given point outside it, a complement of a set with nullity $2$, which we found above.
Hence $\sigma$ is not minimal in $N_2$.

If $|E\setminus \sigma|=4$, all arguments are the same as in the case $5 \le |E\setminus \sigma| \le q$, except in the case when $E\setminus \sigma$ lies on an irreducible cubic. In this case, $n(\sigma)=2$. Indeed, we can not have all $4$ points, and not even $3$ of them, on a line, by Bezout's theorem, and since $5$ points determines the conic uniquely, the $4$ points must impose $4$ independent conditions on the $6$-dimensional affine system of conics. 
Hence we conclude that with $E\setminus \sigma$ being $4$ points, no $3$ of which are on a line, $n(\sigma)=2$. Now to check minimality, remove a point from $\sigma$, i.e. add a fifth point to $E\setminus \sigma$. It is clear that passing through these $5$ points give $5$ independent conditions on conics: 
If a line pair contains them, we must have $3$ of them on a line, and no irreducible conic  contains them, and the line pair is uniquely determined since there are no 4 points on one line. If no line pair contains them, then it also holds that two 
irreducible conics cannot contain them, by Bezout's theorem again. Hence $\sigma$ is a minimal set in $N_2$. There are $\frac{q^2(q^2-1)(q^2-q)(q^2-3q+3)}{24}$ choices of such quadrilaterals $E\setminus \sigma$, and therefore of such minimal $\sigma$.

If $|E\setminus \sigma| \le 3$, i.e., $\sigma \ge q^2-3$,  we utilize the fact that $r(E)=q^2-6$ to conclude $$
n(\sigma) = |\sigma|-r(\sigma) \ge |\sigma| -r(E)\ge (q^2-3)-(q^2-6)=3.
$$
Thus $\sigma$ cannot be in $N_2.$
This settles the $N_2$-part of the proof.

Let us determine the inclusion-minimal elements of $N_3.$ Note that the nullity of $\sigma$ is at most $2$ if $E\setminus \sigma$ is of cardinality at least $q+1$, as we saw in the analysis of the $N_2$-case above. 
On the other hand, if $E\setminus \sigma$ has $q$ points on a line, then $n(\sigma)=3$ by Lemma~\ref{help}. Since adding a point to $E\setminus \sigma$ reduces the nullity of $\sigma$ to $2$, the set $\sigma$ of the complement of the $q$ points on the line is indeed minimal in $N_3$. 
There are $q^2+q$ such lines in $\mathbb{A}_q^2.$

Let $4 \le |E\setminus \sigma| \le q-1$. If $E\setminus \sigma$ is on a line, then $\sigma$ cannot be minimal in $N_3$. This is because it (although having nullity $3$) is strictly contained in the set of {\it all} points on the line, which we have seen is the complement of a set in $N_3$.
If the points of $E\setminus \sigma$ are not on a line, then they impose at least $4$ independent conditions on conics, and $n(\sigma) \le 2$, as we saw in the analysis in the $N_2$-case
($3$, but not $4$ points on a line give a fixed line and a line through a fixed point, while a configuration of $4$ points, no three of them on a line, also give $4$ independent conditions on conics, as we saw above).

If $|E\setminus \sigma| \le 2$, then $n(\sigma) \ge  |\sigma| - r(E) \ge (q^2-2)-(q^2-6)=4$. Thus the only case  remains to be studied is $|E\setminus \sigma|=3$. Since $3$ points always impose independent conditions on conics, we have $n(\sigma)=3$, regardless of the condition that these $3$ points are collinear or not. But if they are collinear, they are contained in the set of $q$ points on a line, a set which is the complement of a set of nullity $3$, as we saw above. Hence $\sigma$ is not minimal in $N_3$ in this case. But the complement $\sigma$  of a set $E\setminus \sigma$ of  $3$ non-collinear points both has nullity $3$ and is minimal in $N_3$, since, as we have seen, we add a fourth point, we get $4$ independent conditions on conics. Clearly, there are $\frac{q^2(q^2-1)(q^2-q)}{6}$ choices of $E\setminus \sigma$, and therefore of $\sigma$.

Since we have $n(\sigma)=3$, for all sets $\sigma$ of cardinality $q^2-3$, and nullity $6$ for $E$,
we must have nullity $6-i$ for all sets of cardinality $q^2-i$, for $i=0,1,2,3.$ This implies that all sets of cardinality $q^2-i$ are minimal in $N_{6-i},$ for $i=0,1,2$, and that there are no other (minimal) sets $\sigma$ than these ones  of these nullities. 
Clearly then, there are ${q^2 \choose 2}$ minimal sets in $N_4$, there are $q^2$ minimal sets in $N_5$, and one minimal set $\emptyset$ in $N_6$.

Now assume that $q=3$. Let $\sigma\subseteq E$. If $|E\setminus \sigma|\ge q+2=5$, the same argument as above shows that $n(\sigma)=1$. 
If $|E\setminus \sigma|=q+1=4$, then $E\setminus \sigma$ will be a set of $4$ points with either no $3$ of which are collinear, or $3$ points on one line and $1$ point outside the line. Note that the case of all $4$ points on one line is not possible since a line in $\A_q^2$ has exactly $q=3$ points. One can prove that $\sigma$ in both the cases are of nullity $2$ and are also minimal in $N_2$, with the proof following along the same lines as that in the case of $q\ge 4$, and hence the same number of such configurations. For $|E\setminus \sigma|\le q=3$, the same arguments work to obtain the inclusion-minimal sets of nullity $i$, where $i\ge 2$, and hence the number of configurations. 


{$(c)$} This is trivial, since $M_2$ is the uniform matroid $U(0,4).$
\end{proof}
\begin{remark}
{\rm One observes that for $q=3$, the minimal sets of nullity $2$ are simply the  complements  of any set of $4$ points. Moreover, the minimal sets of nullity $3$ are simply the complements of any set of $3$ points. So sets of the same nullity $i$ have the same cardinality, whenever $i \ge 2$ in this case.
For $q \ge 4$ we get two different cardinalities $q^2-q-1$ and $q^2-4$ for minimal elements of $N_2$, and two different cardinalities $q^2-q$ and $q^2-3$ for minimal elements of $N_3$ (and at least $3$ different cardinalities for minimal elements of $N_1$).}
\end{remark}

\section{Minimal Free Resolutions for $M_q$} \label{minres}
In this section we assume $q \ge 7.$ We will return to the case $q \le 5$ in Section~\ref{smallq}. It may be helpful to recall the description of minimal free resolutions of Stanley-Reisner rings that was reviewed in several paragraphs preceding  Example~\ref{bettiremark}.

\begin{proposition} \label{somebetti}
    Assume that $q\ge 7$. For $0\le j\le q^2$, the 
    $\N$-graded Betti numbers $\beta_{1,j}$ of $C_q$ are as follows:
    \begin{align*}
        \beta_{1,q^2-q+1} \ &= \ \frac{q^3(q^2-1)}{2},
        & \beta_{1,q^2-2q+1} \  &= \  \frac{q^4+q^3}{2}, 
        \\
        \beta_{1,q^2-q} \ &= \ q^4-q^2, 
         & \beta_{1,q^2-2q} \  &= \ \frac{q^3-q}{2},  \\ 
 \beta_{1,q^2-q-1} \ &= \ \frac{q^3(q-1)^2}{2}, \qquad \text{and }  & \beta_{1,j}\  &=  \ 0  \ \text{ for all other } j.
    \end{align*}  
\end{proposition}
\begin{proof}
    Let $j\in \{0.1. \dots, q^2\}$. 
    From Corollary~\ref{betainfo}(a),  we see that $\beta_{1,j} \ne 0$ precisely when $j$ is the cardinality of a circuit, and moreover,  in view of  
    \eqref{eq:VariousBeta}  and Remark~\ref{rem:newlynumbered}(ii), $\beta_{1,j}$ is the number of  inclusion-minimal subsets of $N_1$ with $j$ elements. Now the desired result follows from Theorem \ref{nullity}(a).
\end{proof}

Let $\Delta$ be the simplicial complex formed by the independent sets of the  matroid $M_q$ corresponding to the code $C_q=\RM_q(2,2)$. We know that the length of a minimal free resolution of the Stanley-Reisner ring $R_\Delta $ is $\dim C_q$, which is $6$ (since $q>2$). Thus this resolution looks like 
$$
0 \longleftarrow R_\Delta  \overset{\partial_0}{\longleftarrow} F_0=S  \overset{\partial_1}{\longleftarrow} F_1 \longleftarrow \cdots \overset{\partial_6}{\longleftarrow} F_6\longleftarrow 0, 
$$
where $F_0=S$ and by Proposition~ \ref{somebetti}, 
\begin{eqnarray*}
 & & F_1 = S(-q^2+q-1)^{\frac{q^3(q^2-1)}{2}} \oplus S(-q^2+q)^{q^4-q^2}\oplus  S(-q^2+q+1)^{\frac{q^3(q-1)^2}{2}}  \\  & & \qquad \qquad \oplus \; S(-q^2+2q-1)^{{\frac{q^4+q^3}{2}}
}\oplus S(-q^2+2q)^{{\frac{q^3-q}{2}}}
\end{eqnarray*}
Now, arguing as in the proof of Proposition~\ref{somebetti}, but using Theorem~\ref{prep}(a) and Theorem~\ref{nullity}(b), we see that 
$$
F_2=S(-q^2+q+1)^{(q^4-q^2)\beta_{2,\theta}} \oplus S(-q^2+4)^{\frac{q^2(q^2-1)(q^2-q)(q^2-3q+3)}{24}\beta_{2,\gamma}},
$$
where $\theta$ is a fixed subset of $\A^2_q$ given by the complement of 
 the union of a line and a point outside this line, and $\gamma$ is the complement of a fixed quadrilateral (i.e., a set of $4$ points in $\A^2_q$, no three of which are collinear). Indeed, Theorem~\ref{nullity}(b) tells us how many subsets of $\A^2_q$ of types $\theta$ and $\gamma$ we have. Also, 
 we 
 tacitly used the fact the $\N^E$-graded Betti number $\beta_{2,\sigma}$ is the same as $\beta_{2,\theta}$ or $\beta_{2,\gamma}$ according as $\sigma$ is of type  $\theta$ or $\gamma$. (See Remark~\ref{rem:newremark} below.) 
 In a similar manner, 
 $$
 F_3= S(-q^2+q)^{(q^2+q)\beta_{3,\alpha}} \oplus S(-q^2+3)^{\frac{q^2(q^2-1)(q^2-q)}{6}\beta_{3,\delta}},
 $$
 where $\alpha$ and $\delta$ denote complements of a fixed line in $\A^2_q$ and of a fixed set of $3$ non-collinear points in $\A^2_q$, respectively. Furthermore, 
$$
F_4=S(-q^2+2)^{{\binom{q^2}{2}}\beta_{4,\epsilon}},\quad F_5=S(-q^2+1)^{q^2\beta_{5,\omega}}
\quad \textrm{and} \quad F_6=S(-q^2)^{\beta_{6,E}}, 
$$
where $\epsilon$ and $\omega$ denote the complements of a line pair 
and of a point in $E=\A^2_q$, respectively. 

\begin{remark}\label{rem:newremark}
{\rm 
  If $\sigma, \tau$ are subsets of $E$ of the same cardinality, then it is not true, in general, that $\beta_{i, \sigma} = \beta_{i, \tau}$. To see this, it suffices to  use Theorem~\ref{prep}(a) and find  $\sigma, \tau \subseteq E$ 
  of the same cardinality such that $\sigma$ is inclusion-minimal in $N_i$, but $\tau$ is not. However, 
  if $i\ge 2$ and $\tau \in N_i$ is one of the sets $\theta, \gamma, \alpha, 
  \delta, \epsilon, \omega$  defined above, and  if $\sigma$ is a cycle in $N_i$ of the same ``type" as $\tau$, then $\beta_{i, \sigma} = \beta_{i, \tau}$. This follows from Proposition~\ref{mu} by noting that there is an inclusion-preserving bijection between the posets of subcycles of $\sigma$ and $\tau$, and that the M\"obius number $\mu_M^{ }(\sigma)$ depends only on the inclusion relations of cycles (of lower nullity) contained in $\sigma$. 
  }
\end{remark}

At any rate, the discussion above shows that an $\N$-graded minimal free resolution of the Stanley-Reisner ring for the matroid $M_q$ is given by 
\begin{eqnarray}\label{eq:RMresolution}
& & 0 \longleftarrow R_\Delta  \overset{\partial_0}{\longleftarrow} S  \overset{\partial_1}{\longleftarrow} S(-q^2+q-1)^{\frac{q^3(q^2-1)}{2}} \oplus S(-q^2+q)^{q^4-q^2}\oplus \nonumber \\[.2em] 
& & \qquad S(-q^2+q+1)^{\frac{q^3(q-1)^2}{2}} \oplus S(-q^2+2q-1)^{{\frac{q^4+q^3}{2}}
}\oplus S(-q^2+2q)^{{\frac{q^3-q}{2}}}  \nonumber \\[.2em]
& & \overset{\partial_2}{\longleftarrow} S(-q^2+q+1)^{(q^4-q^2)\beta_{2,\theta}} \oplus S(-q^2+4)^{\frac{q^2(q^2-1)(q^2-q)(q^2-3q+3)}{24}\beta_{2,\gamma}}  \nonumber \\[.2em]
& & \overset{\partial_3}{\longleftarrow} S(-q^2+q)^{(q^2+q)\beta_{3,\alpha}} \oplus S(-q^2+3)^{\frac{q^2(q^2-1)(q^2-q)}{6}\beta_{3,\delta}} \overset{\partial_4} \longleftarrow S(-q^2+2)^{{{q^2}\choose{2}}\beta_{4,\epsilon}} \nonumber \\[.2em]
&  &   \overset{\partial_5}{\longleftarrow} S(-q^2+1)^{q^2\beta_{5,\omega}} \overset{\partial_6}{\longleftarrow} 
S(-q^2)^{\beta_{6,E}}{\longleftarrow} 0.
\end{eqnarray}
It remains to determine the $\N^E$-graded Betti numbers 
$$
\beta_{2,\theta}, \; \beta_{2,\gamma}, \; \beta_{3,\alpha}, \; \beta_{3, \delta}, \; \beta_{4,\epsilon}, \; \beta_{5,\omega}, \; \beta_{6,E},
$$  
or equivalently,  the $\N$-graded Betti numbers 
\begin{equation}\label{eq:Bettiseq}
\beta_{2,q^2-q-1}, \; \beta_{2,q^2-4}, \; \beta_{3,q^2-q}, \; \beta_{3, q^2-3}, \; \beta_{4,q^2-2}, \; \beta_{5,q^2-1}, \; \beta_{6,q^2}.
\end{equation} 
We shall do that presently. But it may be interesting to note the following. 

\begin{remark}
{\rm 
Using \eqref{eq:RMresolution} and Theorem \ref{prep} (c), we readily see that for $q\ge 7$,  the generalized Hamming weights $d_i = d_i(C_q)$ for $1\le i \le 6$
 are given by
$$
d_1=q^2-2q, \; d_2=q^2-q-1, \; d_3=q^2-q, \; d_4=q^2-2, \; d_5=q^2-1, \; d_6=q^2.
$$ 
Such generalized Hamming weights were found for all Reed-Muller codes by Heijnen and Pellikaan \cite{HP}; so this is just a special case of their result. 
}
\end{remark}

In order to determine the 7 unknown Betti numbers in \eqref{eq:Bettiseq}, one can try to use the known values in Proposition~\ref{somebetti} and the Boij-S\"oderberg equations given in  Proposition~\ref{BS-HK} and Corollary~\ref{HK-BS}. But these only give $6$ independent equations. Hence we need to find at least one more $\beta_{i,j}$ using other methods.
 
 \begin{lemma} \label{sigma2}
 Let $\theta$ be as above, viz.,  the complement of 
 the union of a line and a point outside this line. Then 
 $\beta_{2,\theta}=q$. 
 Consequently, $\beta_{2,q^2-q-1} = q^5-q^3$. 
 \end{lemma}
 
 \begin{proof}
  Write $E\setminus \theta = L \cup \{P\}$, where $L$ is a line in $E= \A^2_q$ and $P$ is a point of $E$ outside $L$. 
  Here $n(\theta)=2$ and $|\theta|=q^2-q-1$.  Thus, in view of Proposition~\ref{minwords}, the circuits contained in $\theta$ are the complements of the $q+1$ line pairs consisting of $L$ and a line $L'$ through $P$ (regardless of whether or not $L'$ is parallel to $L$). Proposition~\ref{mu} and the recursion for the M\"obius function $\mu_{M_q}^{ }$ then gives:
$\beta_{2,\theta}=|-(q+1)+1|=q$, 
and consequently
$\beta_{2,q^2-q-1}=q(q^4-q^2)=q^5-q^3.$
 \end{proof}

 \begin{remark}\label{rem:similarbeta2}
 {\rm 
     A careful examination of the proof of Lemma~\ref{sigma2} and a comparison of Proposition~\ref{minwords} and parts (a), (b), (c) of Corollary~\ref{cor:smallq} shows that the result in Lemma~\ref{sigma2} is, in fact, valid for any $q\ge 3$.
     }
 \end{remark}

 
 

We are now in a position to determine all the Betti numbers of $M_q$. As noted in Remark~\ref{newphi}, it would be more efficacious to use equations 
\eqref{eq:BS-phi} to find the $\phi_j$'s first.
 
 \begin{corollary}\label{thephi}
Let $\phi_j$ for $j=0,\ldots,q^2$ be as in Definition \ref{phis} for the simplicial complex $\Delta$ corresponding to the matroid $M_q$. Then 
for $q\ge 7$ and $0\le i \le 6$, all the nonzero values of the $\beta_{i,j}$, and all the nonzero values of $\phi_j$ 
are as follows:
\begin{eqnarray*}
&& \phi_{0}=\beta_{0,0} =1,\textrm{ }  \phi_{q^2-2q}=-\beta_{1,q^2-2q}=-\frac{q^3-q}{2}, \textrm{ } \\
&& \phi_{q^2-2q+1}=-\beta_{1,q^2-2q+1}=-\frac{q^4+q^3}{2},\\
&&\phi_{q^2-q-1}=-\beta_{1,q^2-q-1}+\beta_{2,q^2-q-1}=\frac{q^5+2q^4-3q^3}{2},\textrm{ and }\beta_{2,q^2-q-1}=q^5-q^3,\\
&&\phi_{q^2-q}=-\beta_{1,q^2-q}-\beta_{3,q^2-q}=-q^5+3q^3-2q \textrm{ and } \\[.5em]
&& \qquad \qquad \beta_{3,q^2-q}=q^5-q^4-3q^3+q^2+2q,\\
&&\phi_{q^2-q+1}=-\beta_{1,q^2-q+1}=-\frac{q^5-q^3}{2},\\
&&\phi_{q^2-4}=\beta_{2,q^2-4}=\frac{q^9-4q^8+5q^7+q^6-6q^5+3q^4}{24},\\
&&\phi_{q^2-3}= -\beta_{3,q^2-3}=-\frac{q^9-5q^8+8q^7-9q^5+5q^4}{6},\\
&&\phi_{q^2-2}=\beta_{4,q^2-2}=\frac{q^9-6q^8+13q^7-5q^6-14q^5+11q^4}{4},\\
&&\phi_{q^2-1}=-\beta_{5,q^2-1}=-\frac{q^9-7q^8+20q^7-20q^6-15q^5+30q^4-9q^3}{6},\\
&&\phi_{q^2}=\beta_{6,q^2}=\frac{q^9-8q^8+29q^7-51q^6+18q^5+59q^4-60q^3+36q-24}{24}.
\end{eqnarray*}
 \end{corollary}
 
 \begin{proof}
 We find $\phi_j$ for $j=0, q^2-2q, q^2-2q+1, q^2-q-1, q^2-q+1$, using the information from Proposition~\ref{somebetti} and 
Lemma~\ref{sigma2}, and the  resolution \eqref{eq:RMresolution}. We then  use the $6$ linearly independent  Boij-Söderberg equations, as expressed in formula  \eqref{eq:BS-phi},  to find the $\phi_j$, for the remaining  $6$ values, viz., $j=q^2-q,q^2-4,q^2-3,q^2-2,q^2-1, q^2$,
using, for example, SageMath. Having found all the $\phi_j$, we may use the information that we  already have from Proposition~\ref{somebetti} and 
Lemma~\ref{sigma2}, and the resolution \eqref{eq:RMresolution}, to find the 
Betti numbers.
 \end{proof}
 
 \section{Minimal Free Resolutions for the Elongations of $M_q$} \label{minreselong}
In this section, we consider the elongations $M_q^{(\ell)}$ of the matroid $M_q = (E, \mathcal{I}_q)$ corresponding to $C_q$, and determine its minimal free resolutions. Note that $\rk(M_q)= q^2-6$ and thus we may restrict $\ell$ to $\{0,1, \dots , 6\}$. We will assume $q\ge 7$ and use 
the notations in Definition~\ref{phis} with $M=M_q$ throughout this section. 


The key observation we will utilize in this section is the following.
 
 
 

 \begin{lemma} \label{minelong}
 For $0\le i \le q^2-6$, let $N_i$ be as in Definition~\ref{def:nullsetsetc} for the matroid $M_q = (E, \mathcal{I}_q)$, where $E=\A^2_q$. 
 For $0\le \ell \le 6$ and $0\le i \le q^2 - 6 - \ell$, let $N_i^{(\ell)}$ be the set of subsets of $E$ of  nullity $i$ for the $\ell$-th elongation matroid $M_q^{(\ell)}$ of $M_q$. Then the inclusion-minimal elements of  $N_{i}^{(\ell)}$, and thus the only subsets $\sigma \subseteq E$ that give rise to nonzero 
 $\beta_{i,\sigma}^{(\ell)}$, 
 are the inclusion-minimal elements of $N_{i+\ell}$, for  $i=1,\ldots,q^2 - 6 - \ell$. 
 \end{lemma} 
 \begin{proof}
We have seen in the proof of Corollary \ref{pureresol}(b) that $N_i^{(\ell)}=N_{i+\ell}$ for any $i=1,\ldots,q^2 - 6 - \ell$. This readily implies the desired result. 
 \end{proof}
 
 For a minimal $\N^E$-graded resolution of $R_{\Delta_1}$, 
we can use the same minimal sets of $N_2,\ldots,N_6$ (referring to $M_q$) that we have  already found, but the Betti numbers, now called $\beta_{i,\sigma}^{(1)}$,  will in general not be the same as the $\beta_{i+1,\sigma}$.
For instance, it is clear that all the $\beta_{1,\sigma}^{(1)}$ will be $1$ for the $\sigma$ in the original $N_2=N_1^{(1)}$, since they are now circuits for the new matroid. 
Nonetheless, by arguing as in Section~\ref{minres}, we readily see that 
a minimal free resolution of the Stanley-Reisner ring $R_{\Delta_1}$ of $M_q^{(1)}$ 
is of the following form:
\begin{eqnarray}\label{eq:RMresolutionM1}
& & 
0 \longleftarrow R_{\Delta_1}  \overset{\partial_0^{(1)}}{\longleftarrow} S \overset{\partial_1^{(1)}}{\longleftarrow} S(-q^2+q+1)^{q^4-q^2} \oplus S(-q^2+4)^{\frac{q^2(q^2-1)(q^2-q)(q^2-3q+3)}{24}} \qquad  \nonumber \\[.2em] 
& & \  \overset{\partial_2^{(1)}}{\longleftarrow} 
 S(-q^2+q)^{(q^2+q)\beta_{2,\alpha}^{(1)}} \oplus S(-q^2+3)^{\frac{q^2(q^2-1)(q^2-q)}{6}\beta_{2,\delta}^{(1)}} \overset{\partial_3^{(1)}} \longleftarrow S(-q^2+2)^{{\binom{q^2}{2}}\beta_{3,\epsilon}^{(1)}}  \nonumber \\[.2em] 
& & \ \overset{\partial_4^{(1)}}{\longleftarrow} S(-q^2+1)^{q^2\beta_{4,\omega}^{(1)}} \overset{\partial_5^{(1)}}{\longleftarrow} 
S(-q^2)^{\beta_{5,E}^{(1)}}{\longleftarrow} 0,   
\end{eqnarray}
where $\alpha, \delta, \epsilon$ and $\omega$ are as in Section~\ref{minres}. 
Here we have $5$ unknown $\N^E$-graded Betti numbers $\beta_{2,\alpha}^{(1)}, \;
\beta_{2,\delta}^{(1)}, \;
\beta_{3,\epsilon}^{(1)}, \;
\beta_{4,\omega}^{(1)}$, and 
$\beta_{5,E}^{(1)}$, or equivalently, $5$ unknown $\N$-graded Betti numbers or  $5$ unknown $\phi^{(1)}_{j}$, namely, 
$$
\beta^{(1)}_{2,q^2-q}, \; 
    \beta^{(1)}_{2,q^2-3}, \; \beta^{(1)}_{3,q^2-2}, \; \beta_{4,q^2-1}^{(1)}, \; \beta^{(1)}_{5,q^2} \quad \text{or} \quad 
    \phi^{(1)}_{q^2-q}, \;
     \phi^{(1)}_{q^2-3}, \;
      \phi^{(1)}_{q^2-2}, \;
       \phi^{(1)}_{q^2-1}, \;
        \phi^{(1)}_{q^2}.
$$
Thankfully, there are $5$ Boij-S\"oderberg equations of the form \eqref{eq:BS-phi} for the unknown $\phi^{(1)}_{j}$.  Solving these equations in SageMath proves the following. 

\begin{corollary}\label{thephi1}
For $q \ge 7$, the nonzero values of $\N$-graded Betti numbers $\beta^{(1)}_{i, j}$ and the corresponding $\phi^{(1)}_j$ for the first elongation $M_q^{(1)}$ are as follows. 
\begin{eqnarray*}
&&
\phi^{(1)}_{0} =\beta^{(1)}_{0, 0} = 1, \; \; 
\phi^{(1)}_{q^2-q-1}=-\beta^{(1)}_{1,q^2-q-1}=-(q^4-q^2),\\
&&\phi^{(1)}_{q^2-4}=-\beta^{(1)}_{1,q^2-4}=
-\frac{q^2(q^2-1)(q^2-q)(q^2-3q+3)}{24},\\
&&\phi^{(1)}_{q^2-q}=\beta^{(1)}_{2,q^2-q}=q^4-2q^2-q,\\
&&\phi^{(1)}_{q^2-3}=\beta^{(1)}_{2,q^2-3}=\frac{q^8-4q^7+7q^6-q^5-8q^4+5q^3}{6},\\
&&\phi^{(1)}_{q^2-2}=-\beta^{(1)}_{3,q^2-2}=-\frac{q^8-4q^7+9q^6-7q^5-10q^4+11q^3}{4},\\
&&\phi^{(1)}_{q^2-1}=\beta_{4,q^2-1}^{(1)}=\frac{q^8-4q^7+11q^6-17q^5-6q^4+27q^3-6q^2}{6},\\
&&\phi^{(1)}_{q^2}=-\beta^{(1)}_{5,q^2}=-\frac{q^8-4q^7+13q^6-31q^5+10q^4+59q^3-48q^2-24q+24}{24}.
\end{eqnarray*} 
\end{corollary}

Using the same principles, 
we see that a minimal free resolution of the Stanley-Reisner ring $R_{\Delta_2}$ of $M_q^{(2)}$ 
is of the following form:
\begin{eqnarray*}
&& 0 \longleftarrow R_{\Delta_2}  \overset{\partial_0^{(2)}}{\longleftarrow} S \overset{\partial_1^{(2)}}{\longleftarrow} 
S(-q^2+q)^{(q^2+q)} \oplus 
S(-q^2+3)^{\frac{q^2(q^2-1)(q^2-q)}{6}}
\\[.2em]
&& \quad \overset{\partial_2^{(2)}} \longleftarrow S(-q^2+2)^{{q^2 \choose  2}\beta_{2,\epsilon}^{(2)}} 
\overset{\partial_3^{(2)}}{\longleftarrow} S(-q^2+1)^{q^2\beta_{3,\omega}^{(2)}} \overset{\partial_4^{(2)}}{\longleftarrow} 
S(-q^2)^{\beta_{4,E}^{(2)}}{\longleftarrow} 0, 
\end{eqnarray*}
%
where $\epsilon$ and $\omega$ are as in Section~\ref{minres}.
This time we have $3$ unknown $\N^E$-graded Betti numbers $\beta_{2,\epsilon}^{(2)}, \; \beta_{3,\omega}^{(2)}$ and 
$\beta_{4,E}^{(2)}$, or equivalently, $3$ unknown $\N$-graded Betti numbers 
$ \beta^{(2)}_{2,q^2-2}, \; \beta^{(2)}_{3,q^2-1}, \; \beta^{(2)}_{4,q^2}$, or  $3$ unknown $\phi^{(2)}_{j}$. 
 Using the Boij-S\"oderberg equations \eqref{eq:BS-phi} for $\phi^{(2)}_{j}$ and 
 SageMath, we then obtain: 

\begin{corollary}\label{thephi2}
For $q \ge 7$, the nonzero values of $\N$-graded Betti numbers $\beta^{(2)}_{i, j}$ and the corresponding $\phi^{(2)}_j$ for the second elongation $M_q^{(2)}$ are as follows.
\begin{eqnarray*}
&&
\phi^{(2)}_{0} =\beta^{(2)}_{0, 0} = 1, \; \; \phi^{(2)}_{q^2-q}=-\beta^{(2)}_{1,q^2-q}=-(q^2+q),\\
&&\phi^{(2)}_{q^2-3}=-\beta^{(2)}_{1,q^2-3}=
-\frac{q^6-q^5-q^4+q^3}{6},\\
&&\phi^{(2)}_{q^2-2}=\beta^{(2)}_{2,q^2-2}=\frac{q^6-q^5-q^4+q^3}{2},\\
&&\phi^{(2)}_{q^2-1}=-\beta^{(2)}_{3,q^2-1}=-\frac{q^6-q^5-q^4-q^3}{2},\\
&&\phi^{(2)}_{q^2}=\beta^{(2)}_{4,q^2}=\frac{q^6-q^5-q^4-5q^3+6q^2+6q-6}{6}.
\end{eqnarray*}

\end{corollary}

 Next, we see in a similar manner that a minimal free resolution of the Stanley-Reisner ring $R_{\Delta_3}$ of $M_q^{(3)}$ 
 is of the following form:
$$
0 \longleftarrow R_{\Delta_3}  \overset{\partial_0^{(3)}}{\longleftarrow} S \overset{\partial_1^{(3)}}{\longleftarrow}  S(-q^2+2)^{{q^2 \choose 2}}  \overset{\partial_2^{(3)}}{\longleftarrow} 
S(-q^2+1)^{q^2\beta_{2,\omega}^{(3)}} \overset{\partial_3^{(3)}}{\longleftarrow} 
S(-q^2)^{\beta_{3,E}^{(3)}}{\longleftarrow} 0,
$$
where $\omega$ as in Section~\ref{minres}. This time, instead of using Boij-S\"oderberg equations and SageMath, we just observe that the resolution above is pure and linear. Thus, it is identical with the minimal free resolution of the uniform matroid $U(q^2-3, q^2)$. In fact, one can also see directly that $M_q^{(3)}$ coincides with $U(q^2-3, q^2)$. Hence,
in view of Example~\ref{ex:unifelong}, 
$M_q^{(3+s)} = U(q^2-3+s, q^2)$ for $0\le s \le 3$. Formulas for the Betti numbers of uniform matroids are given in \cite{JV}. Using these, we readily see that the minimal free resolutions of the Stanley-Reisner rings $R_{\Delta_\ell}$ of $M^{(\ell)}_q$, for $\ell =3,4,5,6$,  are given as follows. 



$$0 \longleftarrow R_{\Delta_3}  \overset{\partial_0^{(3)}}{\longleftarrow} S \overset{\partial_1^{(3)}}{\longleftarrow}  S(-q^2+2)^{{q^2 \choose 2}}  \overset{\partial_2^{(3)}}{\longleftarrow} (-q^2+1)^{q^4-2q^2} \overset{\partial_3^{(3)}}{\longleftarrow} 
S(-q^2)^{\frac{q^4-3q^2+2}{2}}{\longleftarrow} 0, $$ 

$$0 \longleftarrow R_{\Delta_4}  \overset{\partial_0^{(4)}}{\longleftarrow} S  \overset{\partial_1^{(4)}}{\longleftarrow} S(-q^2+1)^{q^2} \overset{\partial_2^{(4)}}{\longleftarrow} 
S(-q^2)^{q^2-1}{\longleftarrow} 0, $$ 

$$0 \longleftarrow R_{\Delta_5}  \overset{\partial_0^{(5)}}{\longleftarrow} S  \overset{\partial_1^{(5)}}{\longleftarrow}  
S(-q^2){\longleftarrow} 0, $$

$$
0 \longleftarrow R_{\Delta_6}  \overset{\partial_0^{(6)}}{\longleftarrow} S {\longleftarrow} 0. 
$$ 
 
This gives:
 \begin{corollary}\label{thephiOthers}
 For $q \ge 7$ and $3\le \ell \le 6$, the nonzero values of $\N$-graded Betti numbers $\beta^{(\ell)}_{i, j}$ and the corresponding $\phi^{(\ell)}_j$ for the $\ell$-th  elongation $M_q^{(\ell)}$ are as follows.
 \begin{eqnarray*}
 \begin{aligned}[c]
\phi^{(\ell)}_{0}&=\beta^{(\ell)}_{0,0}  =1 \text{ for $3\le \ell \le 6$,} \\[1em]
\phi^{(3)}_{q^2-2}&=-\beta^{(3)}_{1,q^2-2}=-\frac{q^4-q^2}{2},\\[1em]
\phi^{(3)}_{q^2-1}&=\beta^{(3)}_{2,q^2-1}=q^4-2q^2,\\[1em]
\phi^{(3)}_{q^2}&=-\beta^{(3)}_{3,q^2}=-\frac{q^4-3q^2+2}{2},
\end{aligned}
\qquad \qquad
\begin{aligned}[c]
& \\[1em]
\phi^{(4)}_{q^2-1}&=-\beta^{(4)}_{1,q^2-1} =-q^2,\\[1em]
\phi^{(4)}_{q^2}&=\beta^{(4)}_{2,q^2} =q^2-1,\\[1em]
\phi^{(5)}_{q^2}&=-\beta^{(5)}_{1,q^2}  =-1.
\end{aligned}
\end{eqnarray*}
 \end{corollary}


\section{Generalized weight polynomials and Higher weight spectra of $C_q$}
\label{sec8}

In this section, we shall prove our main result about the higher weight spectra of $C_q= \RM_q(2,2)$ in the case $q\ge 7$. 
 A key step is the following. 

\begin{proposition}\label{prop:GWPofRM}
For $q \ge 7$, the generalized weight polynomials $P_j(Z)$ of $C_q$, where $0\le j \le q^2$, are as follows.
\begin{eqnarray*}
&& P_0(Z)=1, \\
&&\vspace{0.1mm}\\
&& P_{q^2-2q}(Z)=\frac{q^3-q}{2}(Z-1),\\
&&\vspace{0.5mm}\\
&& P_{q^2-2q+1}(Z)=\frac{q^4+q^3}{2}(Z-1),\\
&&\vspace{0.5mm}\\
&& P_{q^2-q-1}(Z)=(q^4-q^2)Z^2-\left(\frac{q^5+4q^4-3q^3-2q^2}{2}\right)Z+\left(\frac{q^5+2q^4-3q^3}{2}\right), \\ 
&&\vspace{0.5mm}\\
&& P_{q^2-q}(Z)=(q^2+q)Z^3+(-q^4+q^2)Z^2+(q^5+q^4-3q^3-2q^2+q)Z\\
&& \hspace{2.5cm}-(q^5-3q^3+2q),\\
&&\vspace{0.5mm}\\
&& P_{q^2-q+1}(Z)=\frac{q^5-q^3}{2}(Z-1),\\
&&\vspace{0.5mm}\\
&& P_{q^2-4}(Z)=\left(\frac{q^8-4q^7+5q^6+q^5-6q^4+3q^3}{24}\right)Z^2\\
&& \hspace{2.5cm} -\left(\frac{q^9-3q^8+q^7+6q^6-5q^5-3q^4+3q^3}{24}\right)Z\\
&& \hspace{3cm}+\left(\frac{q^9-4q^8+5q^7+q^6-6q^5+3q^4}{24}\right),\\
 \end{eqnarray*}
 \begin{eqnarray*}
&& P_{q^2-3}(Z)=\left(\frac{q^6-q^5-q^4+q^3}{6}\right)Z^3-\left(\frac{q^8-4q^7+8q^6-2q^5-9q^4+6q^3}{6}\right)Z^2\\
&& \hspace{2.5cm}+\left(\frac{q^9-4q^8+4q^7+7q^6-10q^5-3q^4+5q^3}{6}\right)Z\\
&& \hspace{3cm}-\left(\frac{q^9-5q^8+8q^7-9q^5+5q^4}{6}\right), \\
&&\vspace{0.5mm}\\
&& P_{q^2-2}(Z)=\left(\frac{q^4-q^2}{2}\right)Z^4-\left(\frac{q^6-q^5+q^3-q^2}{2}\right)Z^3\\
&& \hspace{2.5cm}+\left(\frac{q^8-4q^7+11q^6-9q^5-12q^4+13q^3}{4}\right)Z^2\\
&& \hspace{2.8cm}-\left(\frac{q^9-5q^8+9q^7+4q^6-21q^5+q^4+11q^3}{4}\right)Z\\
&& \hspace{3cm}+\left(\frac{q^9-6q^8+13q^7-5q^6-14q^5+11q^4}{4}\right), \\
&&\vspace{0.5mm}\\
&& P_{q^2-1}(Z)=q^2Z^5-(q^4-q^2)Z^4+\left(\frac{q^6-q^5+q^4-q^3-4q^2}{2}\right)Z^3\\
&& \hspace{2.5cm}-\left(\frac{q^8-4q^7+14q^6-20q^5-9q^4+24q^3-6q^2}{6}\right)Z^2\\
&& \hspace{2.8cm}+\left(\frac{q^9-6q^8+16q^7-9q^6-32q^5+24q^4+18q^3-6q^2}{6}\right)Z\\
&& \hspace{3cm}-\left(\frac{q^9-7q^8+20q^7-20q^6-15q^5+30q^4-9q^3}{6}\right), \\
&&\vspace{0.5mm}\\
&& P_{q^2}(Z)=Z^6-q^2Z^5+\left(\frac{q^4-q^2}{2}\right)Z^4-\left(\frac{q^6-q^5+2q^4-5q^3-3q^2+6q}{6}\right)Z^3\\
&& \hspace{2cm}+\left(\frac{q^8-4q^7+17q^6-35q^5+6q^4+39q^3-24q^2}{24}\right)Z^2\\
&& \hspace{2.1cm}-\left(\frac{q^9-7q^8+25q^7-38q^6-13q^5+69q^4-q^3-48q^2+12q}{24}\right)Z\\
&& \hspace{2.2cm}+\left(\frac{q^9-8q^8+29q^7-51q^6+18q^5+59q^4-60q^3+36q-24}{24}\right).
\end{eqnarray*}
and all other $P_j(Z)$ are identically equal to zero.
\end{proposition}
\begin{proof}
We use  the following formula, which follows, e.g., from   \cite[Theorem 5.1]{JRV}: 
$$P_j(Z)=\sum_{\ell=0}^{6}(\phi_j^{(\ell)}-\phi_j^{(\ell-1)})Z^{\ell}, \quad \text{for } 0\le j \le q^2,
$$
where, by convention, $\phi_j^{(-1)} =0$ for $0\le j \le q^2$.  
Substituting 
the values of the $\phi_j^{(\ell)}$ that we have found in Corollaries~\ref{thephi}, \ref{thephi1}, \ref{thephi2} and \ref{thephiOthers},   and making 
an elementary calculation, we obtain  the desired result. 
\end{proof}

Now we can determine the higher weight spectra of $C_q$ for $q\ge 7$. 

\medskip

{\it Proof of Theorem~\ref{Fallq}$(a)$.} 
Let $w\in \{0,1,\dots, q^2\}$. 
We use the values of the $P_w(Z)$ from Proposition \ref{prop:GWPofRM}, and the formula:
$$P_w(q^e)=\sum_{r=0}^e A_w^{(r)}\prod_{i=0}^{r-1}(q^e-q^i), \quad \text{for }e\ge 0$$
given for example in \cite{HKM} or \cite[Proposition 6]{J}. Moreover, since each multiplier of $A_w^{(r)}$ in the above 
expression of $P_w(q^e)$ is positive, we see that $A_w^{(r)}$ is zero for all the support weights $w$ for which $P_w(Z)$ is zero. \qed

\medskip

Observe that the support weights $w$, which appear with nonzero $A_{w}^{(r)}$ for some $1\le r\le 6$, are the ones that appear as cardinalities of an inclusion-minimal set in $N_i$ for some $i$. 
We present an alternative explanation of this outcome by the following lemma.

\begin{lemma} \label{support}
The supports of the linear subcodes of $C_q$ are precisely the inclusion-minimal elements of $N_j$, i.e., the cycles of nullity $j$ of the associated parity check matroid $M_q = (E, r)$ for $1 \le j \le k$. 
\end{lemma}

%
\begin{proof}
    For any subset $\sigma \subseteq E$, consider the following subcode of $C_q$
    \[
    C_q(\sigma) = \{c \in C_q \colon \Supp(c) \subseteq \sigma\}.
    \]
    By definition of the rank function $r^*$ of the dual matroid $M^*_q$, $r^*(\sigma) = \dim C_q - \dim C_q(E \setminus \sigma)$ and thus the nullity of $\sigma$ is 
    \begin{align}\label{eq:3}
        n(\sigma) =& |\sigma| - r(\sigma) = r^*(E) - r^*(E \setminus \sigma) = \dim C_q(\sigma).
    \end{align}

Now we assume that $\sigma$ is an inclusion-minimal subset of $N_j$ for some $1 \le j \le k$.
Thus by \eqref{eq:3}, the subcode $C_q(\sigma)$ has dimension $j$. By definition, $\Supp(C_q(\sigma)) \subseteq \sigma$ and $n(\Supp(C_q(\sigma))) = j$. Therefore, if $\Supp(C_q(\sigma))$ is strictly contained in $\sigma$, then it will contradict the inclusion-minimality of $\sigma$ in $N_j$. Thus $\sigma$ is the support of the subcode $C_q(\sigma)$ of $C_q$.

Conversely, suppose 
$\sigma = \Supp(D)$ for some subcode $D$ of $C_q$. Thus $D \subseteq C_q(\sigma)$ and from \eqref{eq:3} it follows that $n(\sigma) = \dim C_q(\sigma)$. If $\sigma$ is not minimal in $N_j$ with $j = \dim C_q(\sigma)$, then $\Supp(C_q(\sigma)) \subsetneq \sigma$. Indeed, as in that case, if $\tau \subsetneq \sigma$ is an inclusion-minimal element in $N_j$ with $j = \dim C_q(\sigma)$, then $n(\tau) = \dim C_q(\tau) = \dim C_q(\sigma)$ will imply $C_q(\sigma) = C_q(\tau)$ and $\Supp(C_q(\sigma)) = \Supp(C_q(\tau)) \subseteq \tau \subsetneq \sigma$. But then $\Supp(D) \subseteq \Supp(C_q(\sigma)) \subsetneq \sigma$ will contradict with our assumption $\sigma = \Supp(D)$. Therefore, $\sigma$ is an inclusion-minimal element in $N_j$ with $j =\dim C_q(\sigma)$.
\end{proof}
\begin{remark}
  {\rm  Fix $0\le j\le n$. From the formula of $P_j(Z)$ in Step (2), we can obtain $P_j(Z)$ if we know $\phi_j^{(\ell)}$ for each $0\le \ell \le k$. Moreover, from the same formula, we can also obtain $\phi_j^{(\ell)}$ for each $\ell$ if we know  $P_j(Z)$. This is because if we write $p_j^{(\ell)}$ as the coefficient of $Z^\ell$ in $P_j(Z)$ for each $\ell$, then we have
$$
\phi_j^{(\ell)}-\phi_j^{(\ell-1)}=p_j^{(\ell)}  \quad \text{for each } 0\le \ell \le k
$$
and we can uniquely solve the equations for all $\phi_j^{(\ell)}$. Thus, for each $j$, the data of $\phi_j^{(\ell)}$ for $0\le \ell\le k$ is equivalent to the data of $P_j(Z)$. }
\end{remark}
\section{The cases $q=5,4,3,2$} \label{smallq}

\subsection{The case $q=5$}
For $q=5$ the minimal sets in $N_i$, for $1,\ldots,6$, can be characterized exactly in the same way as those for $q \ge 7$, except for $i=1$, when  the complement of an irreducible conic intersecting $L$, the line at infinity, in $2$ points, is no longer (a minimal set) in $N_1$. In fact it is a special case of (the complement of)  $4$ points, no three of which are collinear, one of the two
 different configurations that give minimal sets in $N_2$. Moreover, the argument in Lemma \ref{sigma2} for $\beta_{2,\theta}=q$ where $\theta$ is the complement of the union of a line and a point outside this line, is exactly as in the case of  $q \ge 7$. 
 
 With this in view and inserting $q=5$ in the resolution \eqref{eq:RMresolution} with the Betti number from Lemma \ref{sigma2}, we obtain a minimal free resolution of the Stanley-Reisner ring for $M_5$ (or $C_5$) of the following form:
 %



%
\begin{eqnarray*}
&& 0 \longleftarrow R_\Delta  \overset{\partial_0}{\longleftarrow} S  \overset{\partial_1}{\longleftarrow} S(-20)^{600}
\oplus S(-19)^{1000} \oplus S(-16)^{{375}
}\oplus S(-15)^{{60}} \\[.2em]
&& \overset{\partial_2}{\longleftarrow} S(-19)^{3000} \oplus S(-21)^{\beta_{2,21}} \overset{\partial_3}{\longleftarrow}   S(-20)^{\beta_{3,20}} \oplus S(-22)^{\beta_{3,22}} 
\\[.2em]
&& \overset{\partial_4} \longleftarrow S(-23)^{\beta_{4,23}}  \overset{\partial_5}{\longleftarrow} S(-24)^{\beta_{5,24}} \overset{\partial_6}{\longleftarrow} 
S(-25)^{\beta_{6,E}}{\longleftarrow} 0.  
\end{eqnarray*}


We see that there are $9$ possible 
values of $j$ for which $\phi_j \ne 0$. 
These are $15,16,19,20,21,22,23,24$ and $25$.
Compared with the case $ q \ge 7$, we see that  $q^2-(q-1)=21$ and $q^2-4=21$ now coincide (the irreducible conics intersecting $L$ in two points no longer enter the picture, but four points, among which no three are collinear, still do). So there is one less $\phi_j$ to consider.

With only $6$ unknowns (viz., $\phi_j$ for $20\le j\le 25$), the $6$ Boij-S\"oderberg equations of the form \eqref{eq:BS-phi} permit us to determine all the $\phi_j$ and therefore all the $\beta_{i,j}$. 
In particular, $\beta_{3,20}=2160$ ($20$ is the only value of $j$ for which $\phi_j$ depends on more than one $\beta_{i,j}$, as it also depends on $\beta_{1,20}=600$).

For the first elongation, all the arguments in the case of $q \ge 7$ are also valid for $q=5$, and one gets a minimal free resolution of the form:
\begin{eqnarray*}
0 \longleftarrow R_{\Delta_1}  \overset{\partial_0^{(1)}}{\longleftarrow} S \overset{\partial_1^{(1)}}{\longleftarrow} S(-19)^{600}\oplus S(-21)^{6500} \overset{\partial_2^{(1)}}{\longleftarrow} S(-20)^{\beta^{(1)}_{2,20}} \oplus S(-22)^{\beta^{(1)}_{2,22}} \\
\overset{\partial_3^{(1)}}\longleftarrow S(-23)^{\beta^{(1)}_{3,23}} \overset{\partial_4^{(1)}}{\longleftarrow} S(-24)^{\beta^{(1)}_{4,24}} \overset{\partial_5^{(1)}}{\longleftarrow} 
S(-25)^{\beta^{(1)}_{5,25}} {\longleftarrow} 0.   
\end{eqnarray*}
%
Hence  $5$ Boij-S\"oderberg equations of the form  \eqref{eq:BS-phi} in $5$ unknowns will give us all the $\phi^{(1)}_j$ and therefore, all the $\beta^{(1)}_{i,j}$.  

For the $\ell$-th elongations, for $\ell=2,\ldots,6$, one can argue exactly as in the case of $q \ge 7$,  and insert $q=5$ in the formulas found for the $\phi^{(\ell)}_j$, for all $2\le \ell \le 6$.  In this way, we obtain our desired values for these quantities  
also for $q=5$.

We remark that the values of $\N$-graded Betti numbers of $M_q^{(\ell)}$ as well as the values of $\phi^{(\ell)}_j$ are tabulated in Appendix \ref{AppA} in the case $q=5$ and also for lower values of $q$ that are discussed in the next three subsections. 

\smallskip

\noindent{\it Proof of Theorem \ref{Fallq}$(b)$.} 
As in the case $q \ge 7$,  one finds $P_j(Z)$ using the values of $\phi^{(\ell)}_j$ for each $0\le j\le n$ that we have found above, and then uses the formula
$$
P_w(q^e)=\sum_{r=0}^e A_w^{(r)}\prod_{i=0}^{r-1}(q^e-q^i) \quad \text{for } e\ge 0 \text{ and } 0\le w \le q^2,
$$ 
to find the $A^{(r)}_{w}$ for $0\le w \le q^2$ and $0\le r \le 6$. \qed

\begin{remark} \label{one}
{\rm For $q \ge 7$, we have 
$
A^{(1)}_{q^2-q+1}=\frac{q^5-q^3}{2}$, 
$A^{(2)}_{q^2-4}=\frac{q^8-4q^7+5q^6+q^5-6q^4+3q^3}{24}$, 
and 
$A^{(1)}_{q^2-4}=A^{(2)}_{q^2-q+1}=0.
$ 
Observe that ``the nonzero parts of these formulas are merged'' for $q=5$.}
\end{remark}

\subsection{The case $q=4$}
For $q=4$, the minimal sets in $N_i$, for $1,\dots,6$, can be characterized exactly in the same way as those for $q =5$, except for $i=1$, when  the complement of an irreducible conic intersecting $L$, the line at infinity, at exactly $1$ point, is no longer (a minimal set) in $N_1$. In fact, it is a special case of (the complement of)  $4$ points, no three of which are collinear, one of the two
 different configurations that give minimal sets in $N_2$. 
 It can also be mentioned that the case of irreducible conics intersecting $L$ in exactly 2 points, which gave rise to minimal sets in $N_1$, for $q \ge 7$, but not for $q=5$, now give $3$ non-collinear points in $\mathbb{A}^3$, and thus gives rise to a minimal set  in $N_3$. 
 Moreover, the argument that $\beta_{2,\theta}=q$ for $\theta$ the complement of the union of a line and a point outside this line is exactly as in the case of  $q \ge 5$. 
 
 With this in view and inserting $q=4$ in the resolution \eqref{eq:RMresolution} with $\beta_{2,\theta}$, we then obtain a minimal free resolution of the Stanley-Reisner ring associated to the matroid $M_4$ (or the code  $C_4$) of the form:
%
%
%
\begin{eqnarray*}
&& 0 \longleftarrow R_\Delta  \overset{\partial_0}{\longleftarrow} S  \overset{\partial_1}{\longleftarrow}  
 S(-11)^{288} \oplus S(-9)^{160}
\oplus S(-8)^{30} \\
&& \overset{\partial_2}{\longleftarrow} S(-11)^{960} \oplus S(-12)^{\beta_{2,12}} \overset{\partial_3}{\longleftarrow}  S(-12)^{\beta_{3,12}} \oplus S(-13)^{\beta_{3,13}} \\
&& \overset{\partial_4} \longleftarrow S(-14)^{\beta_{4,14}}\overset{\partial_5}{\longleftarrow} S(-15)^{\beta_{5,15}} \overset{\partial_6}{\longleftarrow} 
S(-16)^{\beta_{6,16}}{\longleftarrow} 0.     
\end{eqnarray*}
We see that there are $8$ possible different 
values of $j$ for which 
 $\phi_j \ne 0.$ These are precisely $8,9,11,12,13,14,15$ and $16$.
Compared with $ q \ge 7$, we find  that $q^2-q+1=13$ and $q^2-q=12$ are still there, but they 
no longer correspond to elements of $N_1$. 

With only $5$ unknowns (viz., $\phi_j$ for $j=12,13,14,15, 16$), the first $5$ of the $6$ Boij-S\"oderberg equations of the form \eqref{eq:BS-phi} permit us to determine all the $\phi_j$.
In particular, $\phi_{12}= 2520$.
Note that the support cardinality $12$ is the only value of $j$ for which $\phi_j$ depends on more than one $\beta_{i,j}$. 
Indeed, $\phi_{12} = \beta_{2,12} - \beta_{3,12} $. Thus we need to  determine one of these $\beta_{i,12}$, and it can be done in a slightly greater generality as follows. 

\begin{lemma}\label{lem:beta12}
Assume that $q\ge 3$. Let $\sigma$ be the complement of a fixed line in $E=\A^2_q$. Then $\sigma$ is a cycle in $N_3$ and the $\N^E$-graded Betti number $\beta_{3, \sigma}$ is equal to $q^3 - 2q^2 -q +2 = (q-2) (q^2-1)$. Further, there are exactly $q^2+q$ cycles in $N_3$ of type $\sigma$, and hence $\beta_{3, q^2-q} = (q^2+q)(q-2) (q^2-1) = q^5 - 3q^4-3q^3 +q^2 +2q$.     
\end{lemma}
\begin{proof}
    Let $\mu$ denote the M\"obius function of the matroid $M_q$. A cycle properly contained in $\sigma$ is either a circuit $\tau_1$ or a cycle $\tau_2$ in $N_2$. By Proposition~\ref{minwords} and parts (a), (b), (c) of Corollary~\ref{cor:smallq}, we see that $E\setminus \tau_1$ consists of a pair of lines in $\A^2_q$ that includes $E\setminus \sigma$,  whereas $E\setminus \tau_2$ consists of the line $E\setminus \sigma$ and a point outside it. Thus there are exactly $q^2+q-1$   circuits 
    and $q^2-q$ cycles of type $\tau_2$ in $\sigma$. By Proposition~\ref{mu}, $\mu(\tau_1)= -1$ and in view of  Lemma~\ref{sigma2},  Remark~\ref{rem:similarbeta2}, and Proposition~\ref{mu}, 
    $\beta_{2,\tau_2} =q = \mu(\tau_2)$. So the defining recursion for $\mu$ shows that 
    $$
    \mu(\sigma) = - \left[ (q^2-q)q +(q^2+q-1)(-1) + 1 \right] \quad 
    \text{and} \quad 
    \beta_{3,\sigma} = | \mu(\sigma)| = q^3 - 2q^2 -q +2.
    $$
    The last assertion is now an immediate consequence in view of Remark~\ref{rem:newremark}. 
\end{proof}

In the case $q=4$, Lemma~\ref{lem:beta12} yields  $\beta_{3,12}=600$ and consequently, 
$\beta_{2,12}=3120$. All other $\N$-graded Betti numbers of $M_q$ are readily determined by the $\phi_j$'s and these are tabulated in Appendix~\ref{AppA}.


For the first elongation, all arguments for general $q \ge 7$ are valid also for $q=4$, and one gets a minimal free resolution of the form:
\begin{eqnarray*}
0 \longleftarrow R_\Delta  \overset{\partial_0}{\longleftarrow} S \overset{\partial_1}{\longleftarrow} S(-11)^{240}\oplus S(-12)^{840} \overset{\partial_2}{\longleftarrow}   S(-12)^{\beta^{(1)}_{2,12}} \oplus S(-13)^{\beta^{(1)}_{2,13}}
\\[.5em]
\overset{\partial_3} \longleftarrow S(-14)^{\beta^{(1)}_{3,14}}  \overset{\partial_4}{\longleftarrow} S(-15)^{\beta^{(1)}_{4,15}} \overset{\partial_5}{\longleftarrow} 
S(-16)^{\beta^{(1)}_{5,16}} {\longleftarrow} 0. 
\end{eqnarray*}
%
%
%
Note that there are $6$ possible 
values of $j$ for which 
 $\phi_j \ne 0.$ 
These are $11,12,13,14,15$ and $16$.
Compared with $ q \ge 7$, we see that $q^2-q=12$ and $q^2-4=12$ now coincide, so there is one  less $\phi^{(1)}_j$ to consider. Hence the $\phi^{(1)}_{q^2-q}$-part and the $\phi^{(1)}_{q^2-4}$ from $q\ge 7$ will come together.
The  $5$ Boij-S\"oderberg equations of the form \eqref{eq:BS-phi} in $5$ unknowns will give us all the $\phi^{(1)}_j$ and  therefore, all the $\beta^{(1)}_{i,j}$. 
Note that 
$\beta^{(1)}_{2,12}=220$ and the support cardinality $12$ is the only value of $j$ for which $\phi^{(1)}_j$ depends on more than one $\beta^{(1)}_{i,j}$ as it also depends on  $\beta^{(1)}_{1,12}= 840$.

For the $\ell$-th elongations, for $\ell=2,\dots,6$, one can argue exactly as for $q \ge 7$, and inserting $q=4$ in the formulas found for the $\phi^{(\ell)}_j$, for all $\ell \ge 2$, we get our desired values for these variables also for $q=4$. Now we can proceed to give a proof of our main result for the case $q=4$.

\noindent {\it Proof of Theorem \ref{Fallq}$(c)$.} We find $P_j(Z)$ from the $\phi^{(\ell)}_j$ for each $0\le j\le n$ that we have found above, similar to the case $q \ge 7$ and use the formula, also similar to the case $q \ge 7$, 
$$
P_w(q^e)=\sum_{r=0}^e A_w^{(r)}\prod_{i=0}^{r-1}(q^e-q^i)\quad \text{for }
e\ge 0 \text{ and } 0\le w \le q^2,
$$ 
to determine all the $A^{(r)}_{w}.$\qed

\subsection{The case $q=3$}

For $q=3$, the minimal sets in $N_1$ are  the complements of 
the unions of two lines, and we still have the two different cases where the lines are parallel, and where they meet. Compared with the case $q=4$, the $q+1$ points on an irreducible conic are now the complement of a minimal set in $N_2$ (complement of four points, among which no three are collinear), and no longer in  $N_1$.

The  minimal sets of $N_2$ are still the complements of four points, among which no three are collinear, and the complement of the union of a line and a point outside the line. But now the cardinality of these two sets are the same, namely $q^2-(q+1)=q^2-4=5$. 

The  minimal sets of $N_3$ are still the complements of $q$ points on a line, and the complement of the union of three non-collinear points. But now the cardinality of these two sets are the same, namely $q^2-q=q^2-3=6$. 

This implies that apart from the nullity $1$-part of the minimal free resolution of the relevant Stanley-Reisner ring, which we know, the resolution is pure.
It is of the following form:
\begin{eqnarray*}
&&  
0 \longleftarrow R_\Delta  \overset{\partial_0}{\longleftarrow} S  \overset{\partial_1}{\longleftarrow}  S(-4)^{54} \oplus S(-3)^{12}\overset{\partial_2}{\longleftarrow}  \oplus S(-5)^{\beta_{2,5}}  
\\
&& \overset{\partial_3}{\longleftarrow}   S(-6)^{\beta_{3,6}} \longleftarrow S(-7)^{\beta_{4,8}} \overset{\partial_5}{\longleftarrow} S(-8)^{\beta_{5,8}} \overset{\partial_6}{\longleftarrow} 
S(-9)^{\beta_{6,9}}{\longleftarrow} 0. 
\end{eqnarray*}
So, we get $6$ Boij-S\"oderberg equations of the form \eqref{eq:BS-phi} to determine $5$ unknowns. Solving these, we readily obtain all the $\phi_j$ and therefore, all the $\beta_{i,j}$.

For the $\ell$-th elongations, for $\ell=1,\ldots,6$, we get  pure and linear resolutions, where the Betti numbers and $\phi^{(\ell)}$ can be found easily. 

As an end result, one obtains:

\noindent{\it Proof of Theorem \ref{Fallq}$(d)$.} One finds $P_j(Z)$ using the values of $\phi^{(\ell)}_j$ for each $1\le j\le n$ that we have found above as in the case $q \ge 7$ and use the formula, also as in the case $q \ge 7$, 
$$P_w(q^e)=\sum_{r=0}^e A_w^{(r)}\prod_{i=0}^{r-1}(q^e-q^i)\quad \text{ for } 
e\ge 0 \text{ and } 0\le w \le q^2,
$$ to find all the $A^{(r)}_{w}.$\qed

\begin{remark} \label{four}
{\rm It is interesting to compare our results on $A^{(r)}_w$ in the cases $r=1,2$ with those of \cite[Proposition~2.12]{KM} and the last part in \cite[Lemma~4.3]{KM}. In \cite[Proposition~2.12]{KM}, Kaplan and Matei give a formula for what would be 
$$
X^{q^2} +(q-1)\sum_w A^{(1)}_w X^{q^2-w}Y^w
$$ 
in our language. In the last part of  \cite[Lemma~4.3]{KM}, they give a formula which is equivalent to finding $\sum_w A^{(2)}_w X^{q^2-w}Y^w.$  The interesting observation is that they give unified formulas for all $q \ge 3$ in \cite{KM}, while we have had to treat free resolutions that are different for the cases $q \ge 7$, and three other cases $q=5,4,3$. 
On the other hand, as observed in Corollary~\ref{cor:UnifiedWtEnum} and the preceding paragraph, we can also give unified formulas for the higher weight polynomials
 $$
 W_r(X,Y)=\sum_w A^{(r)}_w X^{q^2-w}Y^w 
 $$
 for all $q\ge 3$ and $1\le r\le 6$, and these are compatible with those of Kaplan and Matei \cite{KM} when $r=2$. 
}
\end{remark}

\subsection{The case $q=2$.}
Here the code is a subcode of $\F_2^4.$
We evaluate quadratic polynomials $xy, \, x(y-1), \, (x-1)y, \, (x-1)(y-1)$ at $(0,0),\, (0,1),\, (1,0),\, (1,1)$, respectively and obtain the codewords
$$
(0,0,0,1),\, (0,0,1,0), \, (0,1,0,0), \, (1,0,0,0),
$$ 
which constitute a basis for $\F_2^4.$ Hence the evaluation map is surjective, and $C_2=\F_2^4.$ 
Thus, by direct inspection, we obtain 
$A^{(0)}_0=1,\, A^{(1)}_1=4, \, A^{(1)}_2=6, \, A^{(1)}_3=4$,  $A^{(1)}_4=1$, $A^{(2)}_2=6, \, A^{(2)}_3=16, \, A^{(2)}_4=13, \, A^{(3)}_3=4, \, A^{(3)}_4=11, \, A^{(4)}_4=1,$
and all other $A^{(r)}_w$ are zero. 
This proves 
Theorem~\ref{Fallq}$(e)$, and thus the proof of Theorem~\ref{Fallq} is now complete. 

As for the determination of Betti numbers, we note that 
 $C_2=\F_2^4$ is clearly an MDS code. 
Thus one can directly use the formulas in \cite[Example 5.2]{JRV} for the Betti numbers of 
$M_2^{(\ell)}$ for $0\le \ell\le 4$. 
Note, however, that the  formulas in \cite{JRV} refer to the Betti numbers of the Stanley-Reisner ideal $I$, and a small modification is needed to adapt them to our case, where we consider the Betti numbers of the Stanley-Reisner ring $S/I$, and this yields 
\begin{equation*}
\beta_{0,0}^{(\ell)}=1  \quad \text{and} \quad \beta_{i,j}^{(\ell)}=\begin{cases}
        \binom{j-1}{s+\ell}\binom{n}{j} & \text{if } i=j-\ell-s, \\
        0 & \text{otherwise}
    \end{cases}
     \quad \text{for } i,j\ge 1,
\end{equation*}
where $s = n-k$, which in this particular case is $0$. 
Thus, the Betti numbers of $M_2^{(\ell)}$ for $0\le \ell\le 4$ are given by the formulas
\begin{equation*}
 \beta_{0,0}^{(\ell)}=1  \quad \text{and} \quad    \beta_{i,j}^{(\ell)}=\begin{cases}
        \binom{j-1}{\ell}\binom{4}{j} & \text{if } i=j-\ell\\
        0 & \text{otherwise}
    \end{cases}
     \quad \text{for } i,j\ge 1.
\end{equation*}

\begin{remark} \label{anothermetod}
{\rm
As noted earlier, our determination of higher weight spectra of the Reed-Muller code $\RM_q(2,2)$ uses a suitable variant of the method used in \cite{JV20} and \cite{JV21} for projective Reed-Muller codes $\PRM_q(2,2)$ and $\PRM_2(2,3)$. Kaipa and Pradhan \cite{KP} used a somewhat different method to deal with the case of $\PRM_3(2,3)$. Here we give a brief comparison of the two methods for the convenience of the reader. 

The rank function of the generator matroid of a code can be obtained by determining the nullities of the subsets for the parity check matroid.
In \cite[Theorem~2.25]{JP}, it is proved that generalized weight enumerator polynomial of a linear code can be obtained from the extended weight enumerator, which is completely determined if we know the number $B_{j,i}$ of subsets of size $j$ and rank $i$ of the generator matroid $M_{G}$ of the code. Following this method in \cite{KP}, the higher weight spectra of projective Reed-Muller code $\textrm{PRM}_3(2,3)$ are determined.

Note that in \cite{KP}, the authors determine the rank $r(J)$ of subsets of their ground set by first determining the flats of the generator matroids. Though the terminology of flats is not used, it follows from the definition of maximal configuration of a fixed rank in \cite[Definition 3.2]{KP} that such configurations basically are the flats of the generator matroid. 
On the other hand, we in the present paper determine the lattice of cycles of the parity check matroid, a lattice which is opposite to the lattice of flats of the generator matroid. The upshot is that using two different methods, one can determine the flats of the generator matroid which further can be used in determining the higher weight spectra and more finer invariants, i.e., the Betti numbers of the code. Both the methods rely on understanding the geometry of the relevant Veronese embeddings. 
}
\end{remark}

\section*{Acknowledgement}

We thank Hugues Verdure for providing the material for Example \ref{bettiremark}. The second named author thanks the Department of Mathematics at IIT Bombay for its hospitality during a visit in August--November 2023 when much of the work in this paper was done. 

\addresseshere

\appendix
\section{Calculations of Betti numbers and the \texorpdfstring{$\phi^{(\ell)}_j$}{phi^l_j} for \texorpdfstring{$q=2,3,4,5$}{q=2,3,4,5}}\label{AppA}
We tabulate below the values of $\phi_j^{(\ell)}$ and the Betti numbers 
$\beta_{i,i+j}^{(\ell)}$
for $q=5,4,3,2$ explicitly, the procedure for which is described in Section~\ref{smallq}. 

\begin{table}[h]
  \makebox[\textwidth]{%
    \begin{tabular}{|p{1.4cm}|p{1.4cm}|p{1.4cm}|p{1.4cm}|p{1.4cm}|p{1.4cm}|p{1.4cm}|}
      \hline
      \diagbox[width=\dimexpr \textwidth/14+5\tabcolsep\relax, height=0.9cm]{$j$ }{$\ell$}
                   &0 &$1$ & $2$ & $3$& $4$& $5$ \\
      \hline
       15 & - 60  & 0 & 0 &0 & 0 & 0 \\
      \hline
      16 & -375 & 0 & 0 &0 &  0& 0\\
      \hline
      19 & 2000 &  -600& 0 &0 & 0 & 0\\
      \hline
      20 & - 2760   & 570  & -30 & 0&0 & 0 \\
      \hline
      21 & 31000 & -6500 & 0 & 0 & 0  &  0\\
      \hline
      22 & -100000  & 30000  & -2000  &0 &  0 & 0 \\
      \hline
      23 & 127500  & -48000 & 6000 & -300 & 0 & 0 \\
      \hline
      24 & -73250  &  32725 &  -5875 & 575 & -25 & 0\\
      \hline
      25 & 15944  & -8196 & 1904 & -276 &24  &  -1 \\
      \hline
     
    \end{tabular}
  }%
\vspace{0.3cm}
    \caption{$\phi_{j}^{(\ell)}$ for $q = 5$}
\end{table}

\begin{table}[h]
  \makebox[\textwidth]{%
    \begin{tabular}{|p{1.2cm}|p{1.2cm}|p{1.2cm}|p{1.2cm}|p{1.2cm}|p{1.2cm}|p{1.2cm}|p{1.2cm}|}
      \hline
      \diagbox[width=\dimexpr \textwidth/18+5\tabcolsep\relax, height=0.9cm]{$j$ }{$i$}
         &0  & 1  & 2   &3 &4 &5 & 6   \\
      \hline
      0  &1  &  0  &  0  & 0  & 0  & 0 & 0  \\
      \hline
      14 &  0 & 60 &  0   &  0  &  0  & 0 & 0  \\
      \hline
      15 & 0  &375 &  0   & 0   &  0  & 0  & 0\\
      \hline
      17 &  0 & 0   &3000 &  2160  & 0   &  0   & 0 \\
      \hline
      18 &  0 &1000&  0   & 0   &  0  & 0&0\\
      \hline
      19 & 0  &600 &31000& 100000   &  127500   &  73250 & 15944\\
      \hline
    \end{tabular}
  }%
  \vspace{0.3cm}
    \caption{$\beta_{i,i+j}^{(0)}$ for $q = 5$}
\end{table}

\begin{table}[h]
  \makebox[\textwidth]{
    \begin{tabular}{|p{1cm}|p{1cm}|p{1cm}|p{1cm}|p{1cm}|p{1cm}|p{1cm}|}
      \hline
      \diagbox[width=\dimexpr \textwidth/26+5\tabcolsep\relax, height=0.9cm]{$j$ }{$i$}
         & 0 & 1 & 2  & 3  & 4& 5   \\
      \hline
      0  & 1 & 0 & 0  &  0 & 0 & 0     \\
      \hline
      18 & 0 & 600 & 570 & 0  &  0 & 0     \\
      \hline
      20 & 0 & 6500 & 30000 & 48000 & 32725  & 8196  \\
      \hline
    \end{tabular}
  }
 \vspace{0.3cm}
    \caption{$\beta_{i,i+j}^{(1)}$ for $q = 5$}
\end{table}

\begin{table}[h]
  \makebox[\textwidth]{
    \begin{tabular}{|p{1cm}|p{1cm}|p{1cm}|p{1cm}|p{1cm}|p{1cm}|}
      \hline
      \diagbox[width=\dimexpr \textwidth/26+5\tabcolsep\relax, height=0.9cm]{$j$ }{$i$}
         & 0 & 1 & 2  & 3  & 4   \\
      \hline
      0  & 1 & 0 & 0  &  0 & 0      \\
      \hline
      19 & 0 & 30 & 0 & 0  &  0     \\
      \hline
      21 & 0 & 2000 & 6000 & 5875 & 1904   \\
      \hline
    \end{tabular}
  }
  \vspace{0.3cm}
    \caption{$\beta_{i,i+j}^{(2)}$ for $q = 5$}
\end{table}

\begin{table}[h]
  \makebox[\textwidth]{
    \begin{tabular}{|p{1cm}|p{1cm}|p{1cm}|p{1cm}|p{1cm}|}
      \hline
      \diagbox[width=\dimexpr \textwidth/26+5\tabcolsep\relax, height=0.9cm]{$j$ }{$i$}
         & 0 & 1 & 2  & 3    \\
      \hline
      0  & 1 & 0 & 0  &  0      \\
      \hline
      22 & 0 & 300 & 575 & 276     \\
      \hline
    \end{tabular}
  }
  \vspace{0.3cm}
    \caption{$\beta_{i,i+j}^{(3)}$ for $q = 5$}
\end{table}

\begin{table}[h]
  \makebox[\textwidth]{
    \begin{tabular}{|p{1cm}|p{1cm}|p{1cm}|p{1cm}|}
      \hline
      \diagbox[width=\dimexpr \textwidth/26+5\tabcolsep\relax, height=0.9cm]{$j$ }{$i$}
         & 0 & 1 & 2      \\
      \hline
      0  & 1 & 0 & 0       \\
      \hline
      23 & 0 & 25 & 24      \\
      \hline
    \end{tabular}
  }
 \vspace{0.3cm}
    \caption{$\beta_{i,i+j}^{(4)}$ for $q = 5$}
\end{table}

\begin{table}[h]
  \makebox[\textwidth]{
    \begin{tabular}{|p{1cm}|p{1cm}|p{1cm}|p{1cm}|}
      \hline
      \diagbox[width=\dimexpr \textwidth/26+5\tabcolsep\relax, height=0.9cm]{$j$ }{$i$}
         & 0 & 1      \\
      \hline
      0  & 1 & 0       \\
      \hline
      24 & 0 & 1       \\
      \hline
    \end{tabular}
  }
  \vspace{0.3cm}
    \caption{$\beta_{i,i+j}^{(5)}$ for $q = 5$}
\end{table}


\begin{table}[h]
    \centering
    \begin{tabular}{|p{1.3cm}|p{1.3cm}|p{1.3cm}|p{1.3cm}|p{1.3cm}|p{1.3cm}|p{1.3cm}|}
    \hline
    \diagbox[width=\dimexpr \textwidth/16+5\tabcolsep\relax, height=0.9cm]{$j$ }{$\ell$}
    & 0 & 1 & 2 & 3 & 4 & 5\\
    \hline
    8 & -30 & 0 & 0 & 0 & 0 & 0\\
    \hline
    9 & -160 & 0 & 0 & 0 & 0 & 0\\
    \hline
    11 & 672 & -240 & 0 & 0 & 0 & 0\\
    \hline
    12 & 2520 & -620 & -20 & 0 & 0 & 0\\
    \hline
    13 & -10080 & 4320 & -480 & 0 & 0 & 0\\
    \hline
    14 & 12480 & -6960 & 1440 & -120 & 0 & 0\\
    \hline
    15 & -6816 & 4624 & -1376 & 224 & -16 & 0\\
    \hline
    16 & 1413 & -1125 & 435 & -105 & 15 & -1\\
    \hline
    \end{tabular}
\vspace{0.3cm}
    \caption{$\phi^{(\ell)}_{j}$ for $q=4$}
\end{table}

\begin{table}[h]
    \centering
    \begin{tabular}{|p{1cm}|p{1cm}|p{1cm}|p{1cm}|p{1cm}|p{1cm}|p{1cm}|p{1cm}|}
    \hline
    \diagbox[width=\dimexpr \textwidth/26+5\tabcolsep\relax, height=0.9cm]{$j$ }{$i$}
    & 0 & 1 & 2 & 3 & 4 & 5 & 6\\
    \hline
    0 & 1 & 0 & 0 & 0 & 0 & 0 & 0\\
     \hline
    7 & 0 & 30 & 0 & 0 & 0 & 0 & 0\\
    \hline
    8 & 0 & 160 & 0 & 0 & 0 & 0 & 0\\
    \hline
    9 & 0 & 0 & 960 & 600 & 0 & 0 & 0\\
    \hline
    10 & 0 & 288 & 3120 & 10080 & 12480 & 6816 & 1413\\
    \hline
    \end{tabular}
\vspace{0.3cm}
    \caption{$\beta^{(0)}_{i,i+j}$ for $q=4$}
\end{table}

\begin{table}[h]
    \centering
    \begin{tabular}{|p{1cm}|p{1cm}|p{1cm}|p{1cm}|p{1cm}|p{1cm}|p{1cm}|p{1cm}|}
    \hline
    \diagbox[width=\dimexpr \textwidth/26+5\tabcolsep\relax, height=0.9cm]{$j$ }{$i$}
    & 0 & 1 & 2 & 3 & 4 & 5 \\
    \hline
    0 & 1 & 0 & 0 & 0 & 0 & 0 \\
     \hline
    10 & 0 & 240 & 220 & 0 & 0 & 0 \\
    \hline
    11 & 0 & 840 & 4320 & 6960 & 4624 & 1125  \\
    \hline
    \end{tabular}
\vspace{0.3cm}
    \caption{$\beta^{(1)}_{i,i+j}$ for $q=4$}
\end{table}

\begin{table}[h]
    \centering
    \begin{tabular}{|p{1cm}|p{1cm}|p{1cm}|p{1cm}|p{1cm}|p{1cm}|p{1cm}|p{1cm}|}
    \hline
    \diagbox[width=\dimexpr \textwidth/26+5\tabcolsep\relax, height=0.9cm]{$j$ }{$i$}
    & 0 & 1 & 2 & 3 & 4  \\
    \hline
    0 & 1 & 0 & 0 & 0 & 0  \\
     \hline
    11 & 0 & 20 & 0 & 0 & 0  \\
    \hline
    12 & 0 & 480 & 1440 & 1376 & 435   \\ 
    \hline
    \end{tabular}
\vspace{0.3cm}
    \caption{$\beta^{(2)}_{i,i+j}$ for $q=4$}
\end{table}


\begin{table}[h]
    \centering
    \begin{tabular}{|p{1cm}|p{1cm}|p{1cm}|p{1cm}|p{1cm}|p{1cm}|p{1cm}|p{1cm}|}
    \hline
    \diagbox[width=\dimexpr \textwidth/26+5\tabcolsep\relax, height=0.9cm]{$j$ }{$i$}
    & 0 & 1 & 2 & 3   \\
    \hline
    0 & 1 & 0 & 0 & 0   \\
     \hline
    13 & 0 & 120 & 224 & 105    \\ 
    \hline
    \end{tabular}
\vspace{0.3cm}
    \caption{$\beta^{(3)}_{i,i+j}$ for $q=4$}
\end{table}

\begin{table}[h]
    \centering
    \begin{tabular}{|p{1cm}|p{1cm}|p{1cm}|p{1cm}|p{1cm}|p{1cm}|p{1cm}|p{1cm}|}
    \hline
    \diagbox[width=\dimexpr \textwidth/26+5\tabcolsep\relax, height=0.9cm]{$j$ }{$i$}
    & 0 & 1 & 2    \\
    \hline
    0 & 1 & 0 & 0    \\
     \hline
    14 & 0 & 16 & 15   \\ 
    \hline
    \end{tabular}
\vspace{0.3cm}
    \caption{$\beta^{(4)}_{i,i+j}$ for $q=4$}
\end{table}

\begin{table}[h]
    \centering
    \begin{tabular}{|p{1cm}|p{1cm}|p{1cm}|p{1cm}|p{1cm}|p{1cm}|p{1cm}|p{1cm}|}
    \hline
    \diagbox[width=\dimexpr \textwidth/26+5\tabcolsep\relax, height=0.9cm]{$j$ }{$i$}
    & 0 & 1    \\
    \hline
    0 & 1 & 0    \\
     \hline
    15 & 0 & 1   \\ 
    \hline
    \end{tabular}
\vspace{0.3cm}
    \caption{$\beta^{(5)}_{i,i+j}$ for $q=4$}
\end{table}


\begin{table}[h]
    \centering
    \begin{tabular}{|p{1.3cm}|p{1.3cm}|p{1.3cm}|p{1.3cm}|p{1.3cm}|p{1.3cm}|p{1.3cm}|}
    \hline
    \diagbox[width=\dimexpr \textwidth/16+5\tabcolsep\relax, height=0.9cm]{$j$ }{$\ell$}
    & 0 & 1 & 2 & 3 & 4 & 5\\
    \hline
    3 & -12 & 0 & 0 & 0 & 0 & 0\\
    \hline
    4 & -54 & 0 & 0 & 0 & 0 & 0\\
    \hline
    5 & 324 & -126 & 0 & 0 & 0 & 0\\
    \hline
    6 & -600 & 420 & -84 & 0 & 0 & 0\\
    \hline
    7 & 540 & -540 & 216 & -36 & 0 & 0\\
    \hline
    8 & -243 & 315 & -189 & 63 & -9 & 0\\
    \hline
    9 & 44 & -70 & 56 & -28 & 8 & -1\\
    \hline
    \end{tabular}
\vspace{0.3cm}
    \caption{$\phi^{(\ell)}_{j}$ for $q=3$}
\end{table}

\begin{table}[h]
    \centering
    \begin{tabular}{|p{1cm}|p{1cm}|p{1cm}|p{1cm}|p{1cm}|p{1cm}|p{1cm}|p{1cm}|}
    \hline
    \diagbox[width=\dimexpr \textwidth/26+5\tabcolsep\relax, height=0.9cm]{$j$ }{$i$}
    & 0 & 1 & 2 & 3 & 4 & 5 & 6\\
    \hline
    0 & 1 & 0 & 0 & 0 & 0 & 0 & 0\\
    \hline
    2 & 0 & 12 & 0 & 0 & 0 & 0 & 0\\
    \hline
    3 & 0 & 54 & 324 & 600 & 540 & 243 & 44\\
    \hline
    \end{tabular}
\vspace{0.3cm}
    \caption{$\beta^{(0)}_{i,i+j}$ for $q=3$}
\end{table}

\begin{table}[h]
    \centering
    \begin{tabular}{|p{1cm}|p{1cm}|p{1cm}|p{1cm}|p{1cm}|p{1cm}|p{1cm}|}
    \hline
    \diagbox[width=\dimexpr \textwidth/26+5\tabcolsep\relax, height=0.9cm]{$j$ }{$i$}
    & 0 & 1 & 2 & 3 & 4 & 5\\
    \hline
     0 & 1 & 0 & 0 & 0 & 0 & 0\\
    \hline
     4 & 0 & 126 & 420 & 540 & 315 & 70   \\
    \hline
    \end{tabular}
\vspace{0.3cm}
    \caption{$\beta^{(1)}_{i,i+j}$ for $q=3$}
\end{table}

\begin{table}[h]
    \centering
    \begin{tabular}{|p{1cm}|p{1cm}|p{1cm}|p{1cm}|p{1cm}|p{1cm}|}
    \hline
    \diagbox[width=\dimexpr \textwidth/26+5\tabcolsep\relax, height=0.9cm]{$j$ }{$i$}
    & 0 & 1 & 2 & 3 & 4 \\
    \hline
     0 & 1 & 0 & 0 & 0 & 0  \\
    \hline
     5 & 0 & 84 & 216 & 189 & 56  \\
    \hline
    \end{tabular}
\vspace{0.3cm}
    \caption{$\beta^{(2)}_{i,i+j}$ for $q=3$}
\end{table}

\begin{table}[h]
    \centering
    \begin{tabular}{|p{1cm}|p{1cm}|p{1cm}|p{1cm}|p{1cm}|p{1cm}|p{1cm}|p{1cm}|}
    \hline
    \diagbox[width=\dimexpr \textwidth/26+5\tabcolsep\relax, height=0.9cm]{$j$ }{$i$}
    & 0 & 1 & 2 & 3  \\
    \hline
     0 & 1 & 0 & 0 & 0  \\
    \hline
     6 & 0 & 36 & 63 & 28  \\
    \hline
    \end{tabular}
\vspace{0.3cm}
    \caption{$\beta^{(3)}_{i,i+j}$ for $q=3$}
\end{table}

\begin{table}[h]
    \centering
    \begin{tabular}{|p{1cm}|p{1cm}|p{1cm}|p{1cm}|}
    \hline
    \diagbox[width=\dimexpr \textwidth/26+5\tabcolsep\relax, height=0.9cm]{$j$ }{$i$}
    & 0 & 1 & 2   \\
    \hline
     0 & 1 & 0 & 0   \\
    \hline
     7 & 0 & 9 & 8   \\
    \hline
    \end{tabular}
    \vspace{0.3cm}
    \caption{$\beta^{(4)}_{i,i+j}$ for $q=3$}
\end{table}

\begin{table}[h]
    \centering
    \begin{tabular}{|p{1cm}|p{1cm}|p{1cm}|}
    \hline
    \diagbox[width=\dimexpr \textwidth/26+5\tabcolsep\relax, height=0.9cm]{$j$ }{$i$}
    & 0 & 1   \\
    \hline
     0 & 1 &  0   \\
    \hline
     8 & 0 & 1   \\
    \hline
    \end{tabular}
\vspace{0.3cm}
    \caption{$\beta^{(5)}_{i,i+j}$ for $q=3$}
\end{table}



\begin{table}[h]
    \centering
    \begin{tabular}{|p{1cm}|p{1cm}|p{1cm}|p{1cm}|p{1cm}|p{1cm}|}
    \hline
    \diagbox[width=\dimexpr \textwidth/26+5\tabcolsep\relax, height=0.9cm]{$j$ }{$i$}
    & 0 & 1 & 2 & 3 & 4 \\
    \hline
    0 & 1 & 4 & 6 & 4 & 1 \\
    \hline
    \end{tabular}
\vspace{0.3cm}
    \caption{$\beta^{(0)}_{i,i+j}$ for $q=2$}
\end{table}

\begin{table}[h]
    \centering
    \begin{tabular}{|p{1cm}|p{1cm}|p{1cm}|p{1cm}|p{1cm}|}
    \hline
    \diagbox[width=\dimexpr \textwidth/26+5\tabcolsep\relax, height=0.9cm]{$j$ }{$i$}
    & 0 & 1 & 2 & 3 \\
    \hline
    0 & 1 & 0 & 0 & 0 \\
    \hline
    1 & 0 & 6 & 8 & 3 \\
    \hline
    \end{tabular}
\vspace{0.3cm}
    \caption{$\beta^{(1)}_{i,i+j}$ for $q=2$}
\end{table}

\begin{table}[h]
    \centering
    \begin{tabular}{|p{1cm}|p{1cm}|p{1cm}|p{1cm}|}
    \hline
    \diagbox[width=\dimexpr \textwidth/26+5\tabcolsep\relax, height=0.9cm]{$j$ }{$i$}
    & 0 & 1 & 2  \\
    \hline
    0 & 1 & 0 & 0  \\
    \hline
    2 & 0 & 4 & 3  \\
    \hline
    \end{tabular}
\vspace{0.3cm}
    \caption{$\beta^{(2)}_{i,i+j}$ for $q=2$}
\end{table}

\begin{table}[h]
    \centering
    \begin{tabular}{|p{1cm}|p{1cm}|p{1cm}|}
    \hline
    \diagbox[width=\dimexpr \textwidth/26+5\tabcolsep\relax, height=0.9cm]{$j$ }{$i$}
    & 0 & 1   \\
    \hline
    0 & 1 & 0   \\
    \hline
    3 & 0 & 1  \\
    \hline
    \end{tabular}
\vspace{0.3cm}
    \caption{$\beta^{(3)}_{i,i+j}$ for $q=2$}
\end{table}


\end{document}